\pdfoutput=1 

\documentclass[oneside,english,12pt,titlepage]{amsbook}
\usepackage[T1]{fontenc}
\usepackage[utf8]{luainputenc}
\usepackage{verbatim}
\usepackage{amsthm}
\usepackage{amstext}
\usepackage{amssymb}
\usepackage{graphicx}

\usepackage{amscd}

\makeatletter

\let\SF@@footnote\footnote
\def\footnote{\ifx\protect\@typeset@protect
    \expandafter\SF@@footnote
  \else
    \expandafter\SF@gobble@opt
  \fi
}
\expandafter\def\csname SF@gobble@opt \endcsname{\@ifnextchar[%]
  \SF@gobble@twobracket
  \@gobble
}
\edef\SF@gobble@opt{\noexpand\protect
  \expandafter\noexpand\csname SF@gobble@opt \endcsname}
\def\SF@gobble@twobracket[#1]#2{}
\theoremstyle{plain}
\newtheorem{thm}{Theorem}[chapter]

\newtheorem{lem}[thm]{Lemma}
\newtheorem{defn}[thm]{Definition}

\newtheorem{conj}[thm]{Conjecture}

\newtheorem*{thm*}{Theorem}
\newtheorem*{cor*}{Corollary}
\newtheorem*{lem*}{Lemma}
\newtheorem*{defn*}{Definition}
\newtheorem*{prop*}{Proposition}
\newtheorem*{conj*}{Conjecture}

\makeatother

\usepackage{babel}

\begin{document}

\pagestyle{plain}

\frontmatter

\begin{titlepage}

{\Huge $C^{1}$- Stable - Manifolds for Periodic }

\vspace{0.5cm}

{\Huge Heteroclinic Chains in Bianchi IX:}

\vspace{0.5cm}

\begin{center}

{\Huge Symbolic Computations and}

\vspace{0.4cm}

{\Huge  Statistical Properties}

\vspace{0.5cm}
\vspace{1cm}

{\Large Johannes Buchner} 

\vspace{0.5cm}

\today

\end{center}

\vspace{0.5cm}

\vspace{2cm}

{\bf  Abstract}

\vspace{0.5cm}

In this paper we study oscillatory Bianchi models of class A and are able to show that for admissible periodic heteroclinic chains in Bianchi IX there exisist $C^{1}$- stable - manifolds of orbits that follow these chains towards the big bang. A detailed study of Takens Linearization Theorem and the Non-Resonance-Conditions leads us to this new result in Bianchi class A. 
\newline
More precisely, we can show that there are no heteroclinic chains in Bianchi IX with constant continued fraction development that allow Takens-Linearization at all of 
their base points. Geometrically speaking, this excludes "symmetric" heteroclinic chains
with the same number of "bounces" near all of the 3 Taub Points - the result shows that
we have to require some "asymmetry" in the bounces in order to allow for Takens Linearization, e.g. by considering admissible 2-periodic continued fraction developments.
\newline
We conclude by discussing the statistical properties of those solutions, including their topological and measure-theoretic genericity.

\end{titlepage}

\newpage{}

\tableofcontents{}

\newpage{}

% the main part of the document starts
\mainmatter

\newpage

\section{Resonances for Periodic Chains in Bianchi IX}
\label{resonances-bix}

We are interested in periodic heteroclinic chains in the
Bianchi IX cosmological model. These can be represented by a Kasner
parameter $u\in\mathbb{R}$ with an infinite periodic continued fraction
representation, e.g. $u=[a,b,c,a,b,c,...]$.

\subsection{Infinite Periodic Continued Fractions}

From the theory of continued fractions, we know that it holds:

\begin{thm*}
$u\in\mathbb{R}$ has an infinite periodic continued fraction representation
$\iff$$u\in\mathbb{R}$ is a ``quadratic irrational''$\iff u$
is a real but irrational root of a quadratic equation with integer
coefficients, i.e. $\exists:c_{1},c_{2},c_{3}:c_{1}+c_{2}u+c_{3}u^{2}=0$
(with $c_{i}\in\mathbb{Z}$$)$.\\
\end{thm*}
Here we are only interested in the direction "$\Rightarrow$", which follows directly for the general formulas for continued fractions in section \ref{general-higher-periods}, see below. The other direction is a bit more elaborate (see e.g. \cite{per54} §19 or \cite{khin49} §10).\\

For the argument we will carry out later, it is of crucial importance that, up to a
common scaling factor $z\in\mathbb{Z}$, there is exactly one quadratic equation satisfied by a quadratic irrational $u$.

%It will allow us to chose the smallest possible coefficients (i.e. such that $|c_{1}|+|c_{2}|+|c_{3}|$, which means in particular that the $c_{i}$ do not have a common factor. 

As this is very important when considering the resonances of the eigenvalues in Bianchi models, we include a proof of this fact here (and we assume that the $c_{i}$ do not have a common factor because we will be interested in the smallest possible coefficients, where this is clear, see section \ref{smallest-resonance}):

\begin{lem}
\label{uniqueness-lemma}
For a (fixed) quadratic irrational $u$, let $c_{i}\in\mathbb{Z},i=1...3$
be s.t. $c_{1}+c_{2}u+c_{3}u^{2}=0$ and $\gcd(c_{i})=1$, i.e. the $c_{i}$ do not have a common factor.
Now assume that  $d_{1}+d_{2}u+d_{3}u^{2}=0$ also holds with $d_{i}\in\mathbb{Z}$.
Then it follows that 
\[
\exists z:d_{i}=z*c_{i}\mbox{, for (\ensuremath{i=1...3)}}\mbox{ with \ensuremath{z\in\mathbb{Z}}}
\]
\end{lem}

\begin{proof}
Multiplying the equation with coefficients $c_{i}$ with $d_{1}$ and
the other one with $c_{1}$ results in the following two equations:

\begin{eqnarray*}
d_{1}c_{1}+d_{1}c_{2}u+d_{1}c_{3}u^{2} & = & 0\\
c_{1}d_{1}+c_{1}d_{2}u+c_{1}d_{3}u^{2} & = & 0
\end{eqnarray*}

Subtracting the second from the first equation leads to

\begin{eqnarray}
u(d_{1}c_{2}-c_{1}d_{2}+(d_{1}c_{3}-c_{1}d_{3})u) & = & 0\label{eq:u-eq}
\end{eqnarray}

and, as $u\ne0$, we conclude that

\[
u=\frac{c_{1}d_{2}-d_{1}c_{2}}{d_{1}c_{3}-c_{1}d_{3}}
\]

if $d_{1}c_{3}-c_{1}d_{3}\ne0$, which leads to a contradiction because
$u\notin\mathbb{Q}$ was assumed.

If, on the other hand, $d_{1}c_{3}-c_{1}d_{3}=0,$ it follows from
(\ref{eq:u-eq}) that also $d_{1}c_{2}-c_{1}d_{2}=0$, which leads
to the conclusion that $\frac{d_{1}}{c_{1}}=\frac{d_{2}}{c_{2}}=\frac{d_{3}}{c_{3}}:=z$
with $\ensuremath{z\in\mathbb{Z}}$. Note that ${z\in\mathbb{Q}}$ would lead to a contradiction because we assumed that the $c_{i}$ do not have a common factor.
\end{proof}

\subsection{The Case of Bianchi IX}

In order to check the (SNC) for the linearized vectorfield at a point
on the Kasner circle, observe that $DX(p)$ is diagonal and that there
are three hyperbolic eigenvalues for all points of the Kasner circle except for the Taub points.

%6 special points that mark the boundaries of the sectors (these include the .

In terms of the Kasner parameter $u$, the following formulas hold for those three eigenvalues (see section \ref{bianchi-defs}):
% (\ref{eq:EVinBIXinU}) ist noch doppelte refernez?!

\begin{equation}
\left(\lambda_{1},\lambda_{2},\lambda_{3}\right)=\left(\frac{-6u}{1+u+u^{2}},\frac{6(1+u)}{1+u+u^{2}},\frac{6u(1+u)}{1+u+u^{2}}\right)\label{eq:EVinBIXinU}
\end{equation}

All 3 hyperbolic eigenvalues are real. A resonance thus means in
this case: $\exists k=(k_{1},k_{2},k_{3}),k_{i}\in\mathbb{Z}$ s.t.

\begin{equation}
k_{1}\lambda_{1}+k_{2}\lambda_{2}+k_{3}\lambda_{3}=0\label{eq:resEV}
\end{equation}
where either all of the $k_{i}$ must have the same sign, or the one of the $k_{i}$ that has a different sign than the other two must be equal to $\pm 1$. Because this "sign condition"
will play an important role later on, let us make the following definition:

\begin{defn*}
A triple $k=(k_{1},k_{2},k_{3}),k_{i}\in\mathbb{Z}$ satisfies the Resonance Sign Condition (RSC)
$\iff$  either all of the $k_{i}$ must have the same sign, or the one of the $k_{i}$ that has a different sign than the other two must be equal to $\pm 1$
\end{defn*}

Only if a triple fulfils the RSC, it qualifies as a coefficient-triple for a resonance that prevents the application of Takens Linearization Theorem. This means that if we can show that resonant coefficients do not fulfill the RSC, they do not matter and Takens-Linearization is still possible. Note that a simple way of showing that the RSC is not satisfied is to show that one coefficient is strictly bigger than one, while a different one is strictly less than minus one.

% the old version of the description
%where either exaclty one of the $k_{i}$ may be equal to $-1$ and the other 2 are positive (case (\ref{eq:stern1}),
% or all of them have to be positive (case (\ref{eq:stern2})). 

\subsection{SNC for Infinite Periodic Heteroclinic Chains}

In preparation for further generalizations to Bianchi-models of class
B, this section is formulated a bit more general that it would be
necessary for discussing only the case of Bianchi IX. As seen above,
the eigenvalues of the linearized vectorfield in BIX for points of
the Kasner circle can be expressed in the Kasner parameter $u$:

\begin{equation}
\lambda_{i}=\frac{l_{1}^{i}+l_{2}^{i}u+l_{3}^{i}u^{2}}{1+u+u^{2}} \label{eq:formulaEV}
\end{equation}

Combining (\ref{eq:resEV}) and (\ref{eq:formulaEV}), one gets

\begin{equation}
k_{1}(l_{1}^{1}+l_{2}^{1}u+l_{3}^{1}u^{2})+k_{2}(l_{1}^{2}+l_{2}^{2}u+l_{3}^{2}u^{2})+k_{3}(l_{1}^{3}+l_{2}^{3}u+l_{3}^{3}u^{2})=0\label{eq:by k}
\end{equation}

or, equivalently,

\begin{equation}
(k_{1}l_{1}^{1}+k_{2}l_{1}^{2}+k_{3}l_{1}^{3})+(k_{1}l_{2}^{1}+k_{2}l_{2}^{2}+k_{3}l_{2}^{3})u+(k_{1}l_{3}^{1}+k_{2}l_{3}^{2}+k_{3}l_{3}^{3})u^{2}=0\label{eq:by u}
\end{equation}
As discussed above, for infinite periodic heteroclinic chains, there
are (up to a common scaling factor) unique coefficients $c_{i}\in\mathbb{Z}$
s.t. 

\begin{equation}
c_{1}+c_{2}u+c_{3}u^{2}=0
\end{equation}

Comparing (\ref{eq:by k}) to (\ref{eq:by u}), one sees that (SNC)
does not hold if $\exists k=(k_{1},k_{2},k_{3})$ as above and $z\in\mathbb{Z}$
s.t.

\begin{equation}
M*\left(\begin{array}{c}
k_{1}\\
k_{2}\\
k_{3}
\end{array}\right)=z*\left(\begin{array}{c}
c_{1}\\
c_{2}\\
c_{3}
\end{array}\right)\label{eq:resonance}
\end{equation}

with

\[
M=\left(\begin{array}{ccc}
l_{1}^{1} & l_{1}^{2} & l_{1}^{3}\\
l_{2}^{1} & l_{2}^{2} & l_{2}^{3}\\
l_{3}^{1} & l_{3}^{2} & l_{3}^{3}
\end{array}\right)
\]
where we will solve (\ref{eq:resonance}) for $(k_{1},k_{2},k_{3})$
in order to check the order of the first resonance.

\subsection{Conclusions for Bianchi IX}

It can be seen easily that the formulas (\ref{eq:EVinBIXinU}) imply
that for Bianchi IX we have

\[
M_{BIX}=\left(\begin{array}{ccc}
0 & 6 & 0\\
-6 & 6 & 6\\
\text{0} & 0 & 6
\end{array}\right)=6*\left(\begin{array}{ccc}
0 & 1 & 0\\
-1 & 1 & 1\\
\text{0} & 0 & 1
\end{array}\right)
\]
 and
\[
M_{BIX}^{-1}=\frac{1}{6}*\left(\begin{array}{ccc}
1 & -1 & 1\\
1 & 0 & 0\\
\text{0} & 0 & 1
\end{array}\right)
\]
Observe that we have a choice of the factor $z$ on the right hand
side of (\ref{eq:resonance}), and that a choice of $z=6$ will result
in an integer resonance with the smallest possible order: 

\begin{equation}
\left(\begin{array}{c}
k_{1}\\
k_{2}\\
k_{3}
\end{array}\right)=\frac{1}{6}*\left(\begin{array}{ccc}
1 & -1 & 1\\
1 & 0 & 0\\
\text{0} & 0 & 1
\end{array}\right)*6*\left(\begin{array}{c}
c_{1}\\
c_{2}\\
c_{3}
\end{array}\right)=\left(\begin{array}{c}
c_{1}-c_{2}+c_{3}\\
c_{1}\\
c_{3}
\end{array}\right)\label{eq:solved for k}
\end{equation}
If the entries of the vector on the right hand side of (\ref{eq:solved for k})
do not have a common factor, then the first resonance will occur at order
$l:=|k_{1}|+|k_{2}|+|k_{3}|=|c_{1}-c_{2}+c_{3}|+|c_{1}|+|c_{3}|$

\subsection{Uniqueness of the Resonance}
\label{smallest-resonance}

For the argument we will carry out later, it is of crucial importance
that we find the order $l$ of the \textbf{first} resonance, meaning
that we can exclude all resonances with order $\tilde{l}<l$. 

%If we are given a coefficient vector $(k_{1}, k_{2}, k_{3})$

In order to do this, we will need the Lemma \ref{uniqueness-lemma} 
on the uniquness of the coefficients for the quadratic equation for quadratic irrationals.

We claim that if we choose the smallest possible coefficients $c_{i}$ for the equation in $u$ (meaning that the $c_{i}$ do not have a common factor), this will lead to the smallest resonance $l:=|k_{1}|+|k_{2}|+|k_{3}|$.

This is true because of the linear dependence of the $k_{i}$ on the $c_{i}$ in (\ref{eq:solved for k}), meaning
that we can exclude all resonances with order $\tilde{l}<l$.

%: An earlier resonance in the eigenvalues would lead to a quadratic equation with smaller coefficients for u, contradicting our choice of the $c_{i}$.

\newpage

\section{Continued Fraction Expansion for Quadratic Irrationals} \label{Details-Section-Formulas}

We will use the following notation for continued fractions:

\[
u=a_{0}+\frac{1}{a_{1}+\frac{1}{a_{2}+\frac{1}{...}}}=:[a_{0},a_{1},a_{2},...]
\]
In this section, we will consider 3 classes of examples, namely $u\in\mathbb{R}$ with
constant, 2-periodic and 3-periodic continued fraction expansions, i.e either
$u=[a,a,...]$ or $u=[a,b,a,b,...]$ or $u=[a,b,c,a,b,c,...]$ for $a,b,c\in\mathbb{N}$. We also recall from section \ref{bianchi-defs} that the Kasner map has the following form:

\begin{eqnarray*}
u & = & \begin{cases}
u-1 & u\in[2,\infty]\\
\frac{1}{u-1} & u\in[1,2]
\end{cases}
\end{eqnarray*}

\subsection{Constant Continued fraction}
\label{1-periodicCFE}

Because of the form of the Kasner-map, starting with $u=[a,a,...]$
will result in the following base-points on the Kasner-circle:

\begin{eqnarray*}
u_{0} & = & [a,a,a,...]\\
u_{1} & = & [a-1,a,a,...]\\
u_{2} & = & [a-2,a,a,...]\\
 & ...\\
u_{a-1} & = & [1,a,a,...]\\
u_{a} & = & [a,a,a,...]\\
 & ...
\end{eqnarray*}
That's why we have to check the Non-Resonance-Conditions at all points
with $u=[m,a,a,...]$ for $m=1...a$. Now note that for $u=[m,a,a,...]$
it holds that

\[
\frac{1}{u-m}-a=u-m
\]

which means that

\[
(m^{2}-am-1)+(a-2m)u+u^{2}=0
\]

resulting in a coefficient vector

\[
\left(\begin{array}{c}
c_{1}\\
c_{2}\\
c_{3}
\end{array}\right)=\left(\begin{array}{c}
m^{2}-am-1\\
a-2m\\
1
\end{array}\right)
\]

Now we can use equation (\ref{eq:solved for k}) to compute the coefficients
for the resonance of the eigenvectors (we set $s=-1$ in order to
match the condition (\ref{eq:stern2})):

\[
\left(\begin{array}{c}
k_{1}\\
k_{2}\\
k_{3}
\end{array}\right)=-1*\left(\begin{array}{c}
c_{1}-c_{2}+c_{3}\\
c_{1}\\
c_{3}
\end{array}\right)=\left(\begin{array}{c}
-m^{2}+(a-2)m+a\\
-m^{2}+am+1\\
-1
\end{array}\right)
\]

\subsection{2-Periodic Continued Fraction Expansion} 

\label{2-periodicCFE}
          
For\\ $u=[a,b,a,b,...]$, we have to check the base-points with $u=[m,b,a,b,a,...]$
with $m=1...a$ and $u=[m,a,b,a,b,...]$ for $m=1...b$. Applying
the same procedure as above, we note that that $u$ satisfies

\[
\frac{1}{\frac{1}{u-m}-a}-b=u-m\mbox{ }\mbox{\,\&\,\ \ensuremath{\frac{1}{\frac{1}{u-m}-b}-a=u-m}}\mbox{ }
\]

when $u=[m,a,b,a,b,...]$ and $u=[m,b,a,b,a,...]$, respectively, and get the following coefficient vectors for $u$:

\[
\left(\begin{array}{c}
c_{1}\\
c_{2}\\
c_{3}
\end{array}\right)=\left(\begin{array}{c}
-am^{2}+abm+b\\
2am-ab\\
-a
\end{array}\right)\mbox{ \& \ensuremath{\left(\begin{array}{c}
c_{1}\\
c_{2}\\
c_{3}
\end{array}\right)=\left(\begin{array}{c}
-bm^{2}+abm+a\\
2bm-ab\\
-b
\end{array}\right)}}
\]

resulting in these coefficient vectors for the eigenvalues (we set
$s=1$ this time):

\[
\left(\begin{array}{c}
k_{1}\\
k_{2}\\
k_{3}
\end{array}\right)=\left(\begin{array}{c}
-am^{2}+(ab-2a)m+ab-a+b\\
-am^{2}+abm+b\\
-a
\end{array}\right)
\]

and 

\[
\left(\begin{array}{c}
k_{1}\\
k_{2}\\
k_{3}
\end{array}\right)=\left(\begin{array}{c}
-bm^{2}+(ab-2b)m+ab+a-b\\
-bm^{2}+abm+a\\
-b
\end{array}\right)
\]

\subsection{3-Periodic Continued Fraction Expansion}

\label{3-periodicCFE}

In complete analogy to the computations above, we find the following
formulas, for the 3 relevant cases. Note that we show the coefficient
vectors for $u$ below, and in all three cases we have to compute
the coefficient vectors for the eigenvalues as done before:

\[
\left(\begin{array}{c}
k_{1}\\
k_{2}\\
k_{3}
\end{array}\right)=\left(\begin{array}{c}
c_{1}-c_{2}+c_{3}\\
c_{1}\\
c_{3}
\end{array}\right)
\]

\paragraph*{u={[}m,b,c,a,...{]} for m=1...a}

\[
\left(\begin{array}{c}
c_{1}\\
c_{2}\\
c_{3}
\end{array}\right)=\left(\begin{array}{c}
m^{2}+mc+m^{2}bc-bm-am-ac-abcm-1\\
abc+a+b-c-2m-2mbc\\
1+bc
\end{array}\right)
\]

\paragraph*{u={[}m,c,a,b,...{]} for m=1...b}

\[
\left(\begin{array}{c}
c_{1}\\
c_{2}\\
c_{3}
\end{array}\right)=\left(\begin{array}{c}
m^{2}+ma+m^{2}ca-cm-bm-ba-abcm-1\\
abc+b+c-a-2m-2mca\\
1+ca
\end{array}\right)
\]

\paragraph*{u={[}m,a,b,c,...{]} for m=1...c}

\[
\left(\begin{array}{c}
c_{1}\\
c_{2}\\
c_{3}
\end{array}\right)=\left(\begin{array}{c}
m^{2}+mb+m^{2}ab-am-cm-cb-abcm-1\\
abc+c+a-b-2m-2mab\\
1+ab
\end{array}\right)
\]

\newpage

\section{Results on Admissibility of Periodic Heteroclinic Chains in Bianchi IX} \label{sub:AnalyticResultsBIX}

In this section, we will concretely check the Sternberg Resonance Conditions for periodic heteroclinic chains in BIX. 

We will prove some general theorems, while more concrete examples can be found in the Appendix \ref{appendix-BIX-NRC}.

\subsection{Constant Continued Fraction Development}

We will give a proof of the fact that there
are no infinite periodic heteroclinic chains with constant continued
fraction development that allow Takens - Linearization at their base
points. More geometrically, this excludes ``symmetric'' heteroclinic
chains with the same number of ``bounces'' near all of the 3 Taub
Points - the result shows that we have to require some ``asymmetry''
in the bounces in order to allow for Takens-Linearization. Below, there is an illustration of 
the heteroclinic chain belonging to $u=[3,3,...]$ which does not allow for Takens-Linearization:

\bigskip
\begin{minipage}{0.7\textwidth}
\centering
\setlength{\unitlength}{0.277\textwidth}
\begin{picture}(3.6,3.6)(-1.3,-1.8)
\put(-1.3 , 1.8 ){\makebox(0,0)[tl]{%
  \includegraphics[width=3.6\unitlength]{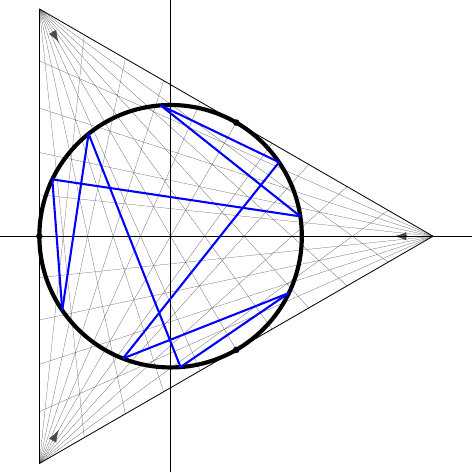}}}
\put( 2.3 ,-0.03){\makebox(0,0)[rt]{$\Sigma_+$}}
\put( 0.03, 1.8){\makebox(0,0)[lt]{$\Sigma_-$}}

\put(-1.03,-0.03){\makebox(0,0)[rt]{$T_1$}}
\put( 0.53, 0.88){\makebox(0,0)[lb]{$T_2$}}
\put( 0.53,-0.88){\makebox(0,0)[lt]{$T_3$}}

\end{picture}
\end{minipage}
\bigskip

\begin{thm}
\label{thm-constant}
For any heteroclinic chain with constant continued
fraction development,  Takens-Linearization fails at some base point.
\end{thm}

\begin{proof}

As we have seen above, a periodic heteroclinic chain has a periodic continued fraction development, leading to a resonance, and let us 
call the coefficients for that resonance $k=(k_{1}, k_{2}, k_{3})$. The first thing we have to check is if k satisfies the Resonance Sign Condition (RSC) defined above.

\begin{lem*}
For constant continued fraction development,  $(u=[a,a,...])$, the coefficient vector $k=(k_{1}, k_{2}, k_{3})$ satisfies the Resonance Sign Condition (RSC) at all base points.
\end{lem*}

\begin{proof}

To prove the Lemma, we observe the following when looking at the formulas for constant continued fraction development in section \ref{1-periodicCFE}:
\begin{itemize}
\item for $m=a$, it holds that $k=(1,a,-1)$
\item for $m=a-1$,  $k=(-a,1,-1)$
\item for $1\ge m < a-1$ and $k=(k_{1}, k_{2}, k_{3})$, it holds that $k_{1}, k_{2} > 0$, while $k_{3}=-1$
\end{itemize}
Thus, the RSC are satisfied in all cases, and the coefficient vector would qualify.

\end{proof}

To prove Theorem \ref{thm-constant}, we have to compare two things:

\begin{itemize}
\item the order of the resonance of the eigenvalues at the basepoints, expressed first in the Kasner-parameter $(u=[a,a,...])$ and then directly in $a$
\item the required SNC for $C^1$-stable-manifolds, i.e. $\alpha(1)$ at all base points
\end{itemize}

The base points of a infinite periodic heteroclinic chain with $u=[a,a,...]$ are
$u=[m,a,...]$ for $m=1...a$. To prove the Theorem, it is enough to show the violation of
the Sternberg Non-Resonance Conditions at one base point. Consider the case $m=a-1$ and start with the formulas for the coefficient vectors, as computed above:

\[
\left(\begin{array}{c}
k_{1}\\
k_{2}\\
k_{3}
\end{array}\right)=
\left(\begin{array}{c}
-m^{2}+(a-2)m+a\\
-m^{2}+am+1\\
-1
\end{array}\right)=
\left(\begin{array}{c}
1\\
a\\
-1
\end{array}\right)
\]

Therefor, it holds that $|k|=a+2$, i.e. we have linear growth of $|k|$ in a.

On the other hand, re-consider the formulas for the eigenvalues in BIX: 

\begin{equation}
\left(\lambda_{1},\lambda_{2},\lambda_{3}\right)=\left(\frac{-6u}{1+u+u^{2}},\frac{6(1+u)}{1+u+u^{2}},\frac{6u(1+u)}{1+u+u^{2}}\right)\label{eq:EVinBIXinU}
\end{equation}

and order them according to magnitude (with the notation from the SNC's from the Takens-Theorem):

\[
\left(\begin{array}{c}
N\\
n\\
m=M
\end{array}\right)=
\left(\begin{array}{c}
|\lambda_{1}|\\
|\lambda_{2}|\\
|\lambda_{3}|
\end{array}\right)
\]

Insert in the formulas for $\alpha, \beta$ and compute:

\[
\beta=Ceiling[\frac{N+k(M+n)}{n}] \ge \frac{u^2+3u+1}{u+1}
\]

\[
\alpha=Ceiling[\frac{M+\beta(N+m)}{m}]\ge \frac{u^3+5u^2+8u+3}{u+1}
\]

This shows quadratic growth for $\alpha$ in u. In fact, for $u=[a-1,a,...]=a-1+\frac{1}{a+\frac{1}{a+\frac{1}{...}}}$, it holds $\forall a>0: |k|<\alpha(1)$, i.e. the SNCs are violated and Takens-Linearization is not possible, which proves the Theorem. For consistency, also compare to Appendix \ref{appendix-BIX-NRC}, where we used Mathematica to compute $\alpha(1)$ and $|k|$ for $u=[m,a,...]$ for $m=1...a$ and $a=1...9$.

% den letzten Satz noch besser formulieren und checken, ob das wirklich alle fälle abdeckt?

\end{proof}

\subsection{2-Periodic Continued Fraction Development}

In this section, we will prove the following Theorem:

\begin{thm}
\label{thm-2per}
For admissible heteroclinic chains with 2-periodic continued fraction development, Takens Linearization is possible at all base points. 
\end{thm}

Here, admissible means that the continued fraction developments has minimal period 2 and the entries are strictly bigger than one (even after cancelling out a possible common factor). To be precise, we define an admissible 2-periodic continued fraction development as follows:

\begin{defn}
\label{adm-2per}
A 2-periodic  $u=[a,b,a,b,...]$ is called admissible $\iff$  $a,b>1$ and neither $a \mid b$ nor $b \mid a$.
\end{defn}

Note that from the condition above, it follows in particular that $a \ne b$, beeing consistent with the results in the section above about constant contiuned fractions.
Two examples of such a heteroclinic chains are illustrated below, with u=[3,2,3,2,...] and with u=[2,3,2,3,...], which are 10-cycles (also compare Appendix \ref{appendix-2periodic}):

% komisch: man darf zwischen minipages keine leerzeile lassen, sonst ned nebeneinander?!
\bigskip
\begin{minipage}[t]{0.5\textwidth}
\includegraphics[width=60mm]{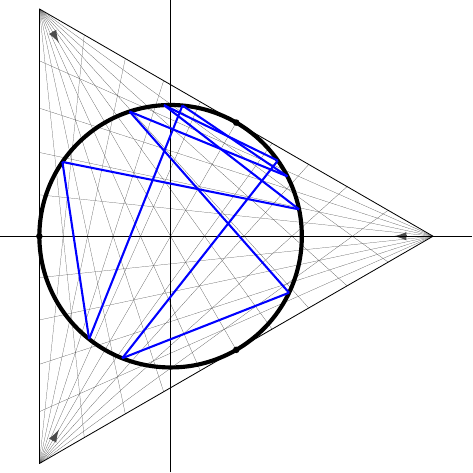}

%muss noch beschriftet werden

%\put( 2.3 ,-0.03){\makebox(0,0)[rt]{$\Sigma_+$}}
%\put( 0.03, 1.8){\makebox(0,0)[lt]{$\Sigma_-$}}

%\put(-1.03,-0.03){\makebox(0,0)[rt]{$T_1$}}
%\put( 0.53, 0.88){\makebox(0,0)[lb]{$T_2$}}
%\put( 0.53,-0.88){\makebox(0,0)[lt]{$T_3$}}

\end{minipage}
\begin{minipage}[t]{0.5\textwidth}
\includegraphics[width=60mm]{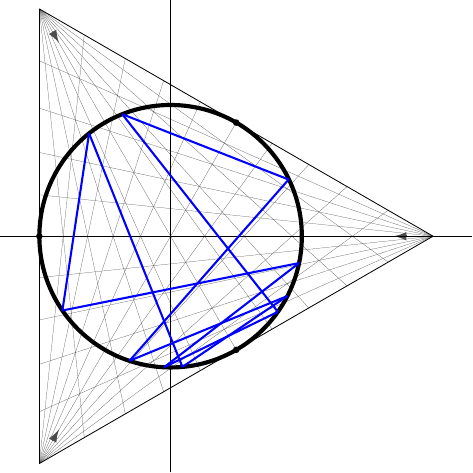}
\end{minipage}
\bigskip

\begin{proof}

The Theorem will directly follow from the following Lemma:

\begin{lem}
\label{lem-2per}
For admissible 2-periodic continued fraction developments, the coefficient vector $k=(k_{1}, k_{2}, k_{3})$ violates the Resonance Sign Condition (RSC) at all base points
\end{lem}

\begin{proof}
When we look at the formulas for 2-periodic continued fraction development in section \ref{2-periodicCFE}, we can observe the following:
\begin{itemize}
\item for $u=[m,a,b,a,b,...]$ and $m=1...b$, it holds that $k_3=-a<-1$ and $k_2\ge b > 1$ as $bm \ge m^2$
\item for $u=[m,b,a,b,a,...]$ and $m=1...a$, it holds that $k_3=-b<-1$ and $k_2\ge a > 1$ as $am \ge m^2$
\end{itemize}
This means that the RSC are violated at all base points of the heteroclinic chain, and the lemma is proven.  Note that we need $a,b>1$, and that if we had $a \mid b$ or $b \mid a$, then coefficients $k_{1}, k_{2}, k_{3}$ would have a common factor we could cancel, leading to an earlier resonance. That's why we need to restrict to admissible 2-periodic continued fraction developments as defined above.

\end{proof}

The Lemma shows that, for "sign reasons", the occuring resonaces are excluded and do not matter for the application of the Takens Theorem. Therefor Takens Linerarization is possible, as claimed in Theorem \ref{thm-2per}.

\end{proof}

\newpage
\subsection{Continued Fraction Development with Higher Periods}
\label{general-higher-periods}

The idea behind the proof of Lemma \ref{lem-2per} can be generalized to continued fraction developments with higher periods. However, it is not so easy anymore to find conditions that assure in general that the resulting coefficients do not have a common factor. We will comment on this matter further at the end of the section.

At first, consider 3-periodic continued fractions. Three examples of such a heteroclinic chains are illustrated below, with u=[1,1,2,1,1,2,...],\\ u=[1,2,1,1,2,1,...] and u=[2,1,1,2,1,1,...] which are 8-cycles and arguably the simplest examples of periodic heteroclinic chains where our method works (this can be checked directly for the concrete examples above, see Appendix \ref{appendix-3periodic}). They all start in sector 5, and the different position of the number "2" in the contiued fraction development leads to bounces around the different Taub points which can be seen in the pictures below:

\bigskip
\begin{minipage}[t]{0.3\textwidth}
\includegraphics[width=45mm]{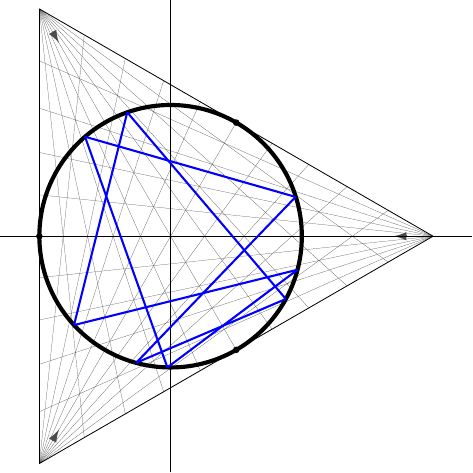}
\end{minipage}
\begin{minipage}[t]{0.3\textwidth}
\includegraphics[width=45mm]{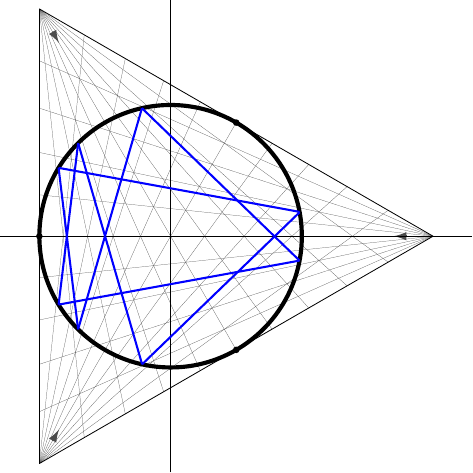}
\end{minipage}
\begin{minipage}[t]{0.3\textwidth}
\includegraphics[width=45mm]{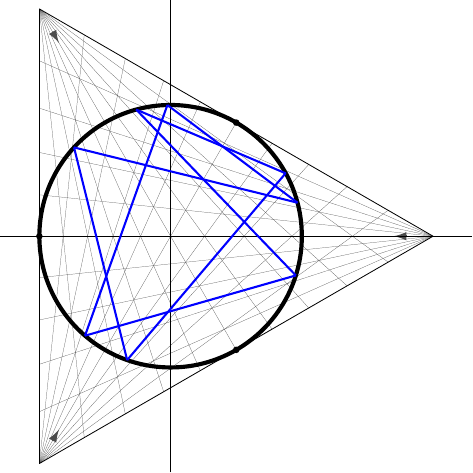}
\end{minipage}
\bigskip

\begin{lem}
\label{thm-3per} %war mal thm, referenz so gelassen aus historischen gründen

Consider a continued fraction development with minimal period 3, i.e. with $u=[a,b,c,a,b,c,...]$ and not $a=b=c$. Then the corresponding coefficient vector $k=(k_{1}, k_{2}, k_{3})$ violates the Resonance Sign Condition (RSC) at all base points if the $k_i$ do not have a common factor .
%(which is e.g. the case if $a=b=c$).
\end{lem}

\begin{proof}
When we look at the formulas for 3-periodic continued fraction development in section \ref{3-periodicCFE}, we can observe the following:
\begin{itemize}
\item for $u=[m,b,c,a,b,c,a,...]$ and $m=1...a$, it holds that  $k_2=c_1 \le -bm-1 < -1$ and $k_3=c_3=1+bc > 1$ 
\item for $u=[m,c,a,b,c,a,b,...]$ and $m=1...b$, it holds that  $k_2=c_1 \le -cm-1 < -1$ and $k_3=c_3=1+ca > 1$ 
\item for $u=[m,a,b,c,a,b,c,...]$ and $m=1...c$, it holds that  $k_2=c_1 \le -am-1 < -1$ and $k_3=c_3=1+ab > 1$ 
\end{itemize}
This means that the RSC are violated at all base points of the heteroclinic chain if we know that neither $k_2 \mid k_3$ nor $k_3 \mid k_2$. This is true in particular if the $k_i$ do not have a common factor as we have assumed for convenience, so Lemma \ref{thm-3per} is proven.

Note that if we had $a=b=c$, then coefficients $k_{1}, k_{2}, k_{3}$ would have a common factor, resulting in an earlier resonance as explained above. Also compare to Appendix \ref{appendix-3periodic} for a consistency check.
\end{proof}

We now try to generalize the argument above to higher periodic continued fractions. In order to do this let us make some general definitions and observations (following \cite{per54} §19\footnote{but note we have a different labelling of the coefficients as we do not consider continued fractions with enumerators different from one}, compare also \cite{khin49} §10):

For continued fractions of the form

\[
u=a_{0}+\frac{1}{a_{1}+\frac{1}{a_{2}+\frac{1}{...}}}=:[a_{0},a_{1},a_{2},...]
\]

we define the following numbers $A_{k}, B_{k}$ recursively:

\begin{eqnarray*}
A_{k}=A_{k-1}a_{k}+A_{k-2}\\
B_{k}=B_{k-1}a_{k}+B_{k-2}
\end{eqnarray*}

with $A_{-1}=1,A_{-2}=0$ and $B_{-1}=0,B_{-2}=1$, leading to $A_{0}=a_{1}$, $A_{1}=a_{0}a_{1}+1$ and
$B_{0}=1$, $B_{1}=a_{1}$.

%\begin{eqnarray*}
%A_{0}=a_{1}, A_{1}=a_{0}a_{1}+1\\ %,  A_{2}=a_{2}(a_{0}a_{1}+1) + a_{0}\\
%B_{0}=1, B_{1}=a_{1} %, B_{2}=a_{2}a_{1}+1 
%\end{eqnarray*}

For an (infinite) continued fraction, we define the ``tails'' as follows:

\[
\xi_{k}:=[a_{k},a_{k+1},...]
\]

Then we have the following general recursion formula for convergent
infinite continued fractions $u=[a_{0},a_{1},...,a_{k-1},\xi_{k}]$ (and $k\ge0)$:

\[
u=\xi_{0}=\frac{A_{k-1}\xi_{k}+A_{k-2}}{B_{k-1}\xi_{k}+B_{k-2}}
\]

, which can be proved by induction.\footnote{ For $k=0$, the formula holds by definition: $\xi_{0}=\frac{A_{-1}\xi_{0}+A_{-2}}{B_{-1}\xi_{0}+B_{-2}}=\frac{\xi_{0}}{1}$. %In the induction step, use the formula for $u=[a_{1},a_{2},...,\xi_{k}]$, which is possible due to the induction hypothesis.
}. Also compare \cite{khin49} §2 and §3.

Now consider pre-periodic continued fractions with pre-period $h$ and minimal period $p$, as made precise in the following definition:

\begin{defn*}
We call $u$ an $h$-pre-periodic continued fraction with pre-period $h$, minimal period $p$ $\iff$ $u=[a_{0},...a_{h-1},\overline{a_{h},a_{h+1},...,a_{h+p-1}}]$ with $a_{\nu}=a_{\nu+p}$$\forall\nu\ge h$ and $\nexists \tilde{p}<p$ s.t. $a_{\nu}=a_{\nu+\tilde{p}}$$\forall\nu\ge h$
\end{defn*}

Note that it also holds that $\xi_{\nu}=\xi_{\nu+p}$$\forall\nu\ge h$. Thus we can get
the following formulas (set $k=h$ and $k=h+p)$:

\[
\xi_{0}=\frac{A_{h-1}\xi_{h}+A_{h-2}}{B_{h-1}\xi_{h}+B_{h-2}}
\]

and

\[
\xi_{0}=\frac{A_{h+p-1}\xi_{h+p}+A_{h+p-2}}{B_{h+p-1}\xi_{h+p}+B_{h+p-2}}=\frac{A_{h+p-1}\xi_{h}+A_{h+p-2}}{B_{h+p-1}\xi_{h}+B_{k+p-2}}
\]

By solving both equations for $\xi_{h}$, we get the following quadratic
equation for $\xi_{0}:$ 

\[
c_{3}\xi_{0}^{2}+c_{2}\xi_{0}+c_{1}=0
\]

with (we abbreviate $g=h+p)$

\begin{eqnarray*}
c_{3} & = & B_{h-2}B_{g-1}-B_{h-1}B_{g-2}\\
c_{2} & = & B_{h-1}A_{g-2}+A_{h-1}B_{g-2}-A_{h-2}B_{g-1}-B_{h-2}A_{g-1}\\
c_{1} & = & A_{h-2}A_{g-1}-A_{h-1}A_{g-2}
\end{eqnarray*}

The formulas above specialize to (for $h=0$, this corresponds to
the formula for periodic continued fractions without pre-period)

\begin{eqnarray*}
c_{3} & = & B_{p-1}\\
c_{2} & = & B_{p-2}-A_{p-1}\\
c_{1} & = & -A_{p-2}
\end{eqnarray*}

and for $h=1$ to\footnote{compare to the formulas for p=1,2,3 presented in section \ref{3-periodicCFE}, as a consistency check}

\begin{eqnarray*}
c_{3} & = & -B_{p-1}\\
c_{2} & = & A_{p-1}+a_{0}B_{p-1}-B_{p}\\
c_{1} & = & A_{p}-a_{0}A_{p-1}
\end{eqnarray*}

Now we are in a position to state the main aim of this section:
\begin{conj}
\label{thm-higher-periodic}
Let $u=[a_{0},a_{1},...]$ be an (infinite) periodic continued
fraction with minimal period $p\ge3$. Then the corresponding heteroclinic chain allows Takens-Linearization at all base points.
\end{conj}

\begin{proof}(idea of proof, but note the remark below)\\
Let $u=[a_{0},a_{1},...]$ be an (infinite) periodic continued fraction.
We need to show that the NRC's are satisfied at all base points of
the heteroclinic chain. Because of the form of the Kasner-map, we have to check all Kasner-parameters of the form $u=u_{m}=[m,\overline{a_{1},a_{2},...,a_{p-1},a_{p}}]$,
i.e. it holds that $a_{0}=m$ (with $1\le m\le a_{p}$) and $a_{\nu}=a_{\nu+p}$,
but now only $\forall\nu\ge1$. From the formulas above (case $h=1)$
we observe the following for the corresponding coefficients of the
resonances of the eigenvalues:

\[
k_{3}=c_{3}=-B_{p-1}<-1
\]
 
where we need our assumption that $p\ge3$ as $B_{1}=a_{1}$ which might be one, but $B_{2}=a_{2}a_{1}+1$ which is bigger than one. Also

\[
k_{2}=c_{1}=A_{p}-a_{0}A_{p-1}=(a_{p}-a_{0})A_{p-1}+A_{p-2}>1
\]

because we know that $a_{0}=m\le a_{p}$ and $A_{1}=a_{0}a_{1}+1$ is bigger than one. That's why the ``Resonance
Sign Condition'' is violated at all base points, and Takens-Linearization
is possible.
\end{proof}

The reason why we don't call the Conjecture above a Theorem is that we are not able to exclude in general that $c_{1}$ divides $c_{3}$ or vice versa, which is essential for the proof above to work out. We believe it is possible to prove this in general for most periodic continued
fraction with minimal period $p\ge3$, probably with a small set of exceptions, but this is an issue for further research. 

%On the other hand, isolated counter-examples are also probable because -> das ist kür!!!
% most if not all
%remains to be done.
%For minimal period $p=3$, this can be excluded from the formulas in section \ref{3-periodicCFE}, and 

\newpage

\section{Details on the Proof for Stable Manifolds}
\label{core-proof}

In this section, we complete the proof of Theorem \ref{thm:main-bix} by showing that there is a $C^{1}$-hyperbolic structure for the return map in Bianchi IX after linearizing at all base points of a heteroclinic chain. This then leads to a $C^{1}$-stable manifold, as claimed.

% Our aim is to show that there exist stable manifolds for periodic heteroclinic chains. 

We procced along the lines and very close to the paper of Béguin \cite{beg10}, but
we adapt the notation to our needs and the situation of a periodic chain that Béguin does not consider. 

Also compare to the papers by Liebscher et al. \cite{lieetal10, lieetal12}, where they work in a Lipschitz-setting without linearizing at the Kasner circle. There, the following return maps are considered 

\[
\Phi_{k}^{return}=\Phi_{k}^{glob}\circ\Phi_{k}^{loc}:\Sigma_{k}^{in}\rightarrow\Sigma_{k+1}^{in}
\]

where the index $k$ stands for the base points on the Kasner circle of the heteroclinic chain, i.e. $\Phi_{k}^{return}$ maps from one In-section to the next. It is shown that those maps satisfy
the necessary cone conditions to allow for a graph-transform on Lipschitz-graphs on a subset of $\Sigma^{in}$ including the origin (which stands for
the heteroclinic orbit). This then leads to the stable manifold result. 

However, like Béguin \cite{beg10}, we will use a collection $\Phi^{return}_B$
of these return maps for all base points of the set $B \subset \mathcal{K}$. We then show that there exists a $C^{1}$- hyperbolic structure for a suitable subset of the corresponding In-sections $\Sigma_{B}^{in}$. This results in a $C^{1}$-stable manifold.

\subsection{Application of Takens Theorem}

Let $B=\{p_{1},...,p_{n}\}$ be the collection of base points on the
Kasner circle of the periodic heteroclinic chain we are looking at.
Then, as we have chosen an admissible periodic chain that satisfies
the necessary Non-Resonance-Conditions by assumption, we can chose co-ordinates
near each point $p_{k}\in B$ such that the vector field has the form
described by the Takens Theorem, i.e. it is essentially linear in
a neighbourhood $U_{p_{k}}$. More precisely, the application of Takens
Linearization Theorem is done in the following form (compare Béguin, p.10): 
\begin{thm}
\label{thm:local-form-vf}
Let $p\in B$ be any point of the set of admissible base points $B$.
Then there exists a Takens-Neighbourhood $U_{p}$ of $p$ in the phase-space
of the Wainwright-Hsu ODEs $\mathcal{W}$ and a $C^{1}$-coordinate-system
on $U_{p}$ such that the Wainwright-Hsu vector field $X^{W}$ can
be written as

\[
X^{W}(x^{c},x^{s},x^{ss},x^{u})=\lambda_{s}(x^{c})x^{s}\frac{\partial}{\partial x^{s}}+\lambda_{ss}(x^{c})x^{ss}\frac{\partial}{\partial x^{ss}}+\lambda_{u}(x^{c})x^{u}\frac{\partial}{\partial x^{u}}
\]
where $\lambda_{ss}(x^{c})<\lambda_{s}(x^{c})<0<\lambda_{u}(x^{c})$
for all $x^{c}$. 
\end{thm}

\begin{proof}
A direct application of the Takens-Theorem \ref{thm:takens-flow} (from chapter \ref{takens-chapter}) gives the existence of a coordinate system $(x^{c},x^{s1},x^{s2},x^{u})$ in $U_{p}$
s.t. $X^{W}$ has the following form in these coordinates:

\[
X^{W}(x^{c},x^{s1},x^{s2},x^{u})=\phi(x^{c})\frac{\partial}{\partial x^{c}}+\sum_{i,j=1}^{2}a_{ij}(x^{c})y^{si}\frac{\partial}{\partial y^{sj}}+b(x^{c})x^{u}\frac{\partial}{\partial x^{u}}
\]
For the vector field $X^{W}$ in the original coordinates, the set
$\mathcal{K}\cap U_{p}$ is the local center-manifold in the neighbourhood
$U_{p}$ at the point p, and it consists of equilibria. As the vector
field above vanishes on $K=\{x^{s1}=x^{s2}=x^{u}=0\}$ and nowhere
else, it follows that $K=\mathcal{K}\cap U_{p}$. This also means
that $\phi\equiv0$ in the neighbourhood $U_{p}$, i.e. there is no
drift at all in the center-direction. Now fix $\{x^{c}=\xi\}$. As
can be seen from the formula above, the vector field $X^{W}(x^{c},x^{s1},x^{s2},z^{u})$
is linear on the restruction to this submanifold. A linear change
of coordinates then diagonalizes the $2 \times 2$-matrix $(a_{ij})_{i,j \in \{1,2\}}$,
as we have 2 distinct real stable eigenvalues of $X^{W}$ at the point
$(\xi,0,0,0)$%
\begin{comment}
hier nochmal kurz erläutern, dass das im orginal-WHsu-VF so ist und
dass die EW unabhängig von der Koordinatenwahl sind
\end{comment}
, and this diagonalization can be done simultaneously, as eigenvalues
and eigendirections depend in a smooth way on $\xi$. Label these
new coordinates $(x^{c},x^{s},x^{ss},x^{u})$ and observe that we
have found the claimed local form of the vector field

\[
X^{W}(x^{c},x^{s},x^{ss},x^{u})=\lambda_{s}(x^{c})x^{s}\frac{\partial}{\partial x^{s}}+\lambda_{ss}(x^{c})x^{ss}\frac{\partial}{\partial x^{ss}}+\lambda_{u}(x^{c})x^{u}\frac{\partial}{\partial x^{u}}
\]

\end{proof}
For the rest of the section, we will use the follwing coordinates: Near the Kasner-circle,
we take the coordinates given by the Takens-Linearization-Theorem,
at each base point $p_{k}$ of the heteroclinic chain, and otherwise,
we stick to the coordinates $N_{i},\Sigma_{+/-}$ of the Wainwright-Hsu-System.
The different coordinate systems give rise to the following metrics: the Riemanian
metric $g_{p}=dx^{c}\wedge dx^{c}+dx^{s}\wedge dx^{s}+dx^{ss}\wedge dx^{ss}+dx^{u}\wedge dx^{u}=(dx^{c})^{2}+(dx^{s})^{2}+(dx^{ss})^{2}+(dx^{u})^{2}$
for the Takes-coordinates in a neighbourhood $U_{p}$ near a point
$p$ of the Kasner circle, and the Riemanian metric $h=dN_{1}^{2}+dN_{2}^{2}+dN_{3}^{2}+d\Sigma_{+}^{2}+d\Sigma_{-}^{2}$.
Later we use a ``global'' Riemannian metric adapted to our set of
base points $B$ by defining $g_{B}$ such that 

\begin{equation}
\label{global-metric}
g_{B}\restriction U_{p}=g_{p} \forall p\in B
\end{equation}
For the local passage, which we will consider next, we are entirely in the neighbourhood $U_{p}$ and can use the ``local'' metric $g_{p}=(dx^{c})^{2}+(dx^{s})^{2}+(dx^{ss})^{2}+(dx^{u})^{2}$
.

\subsection{Local Passage}

Our next step is to deal with the local passage near an equilibrium
of the Kasner circle $\mathcal{K}$. Figure \ref{detail-LocalMap} shows a graphic illustration
of the situation in Bianchi IX - note that we are in the lucky situation
here that the incoming stable eigenvalue is always stronger that the
outgoing unstable eigenvalue, this will change in $BVI_{-\frac{1}{9}}$
that we deal with in chapter \ref{bvi-chapter}.\\

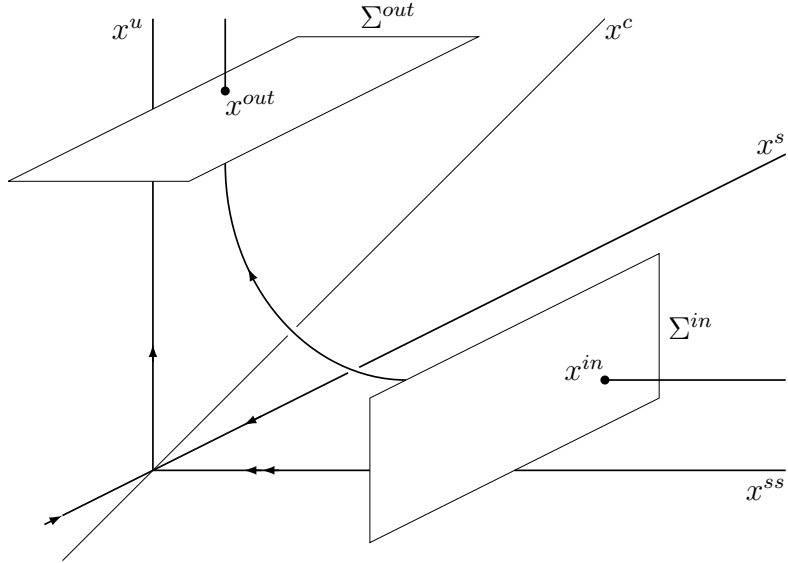
\begin{figure}
\setlength{\unitlength}{0.019\textwidth}
\centering
\begin{picture}(43,30)(-8,-5)
\put(-8,-5){\makebox(0,0)[bl]{\includegraphics[width=43\unitlength]{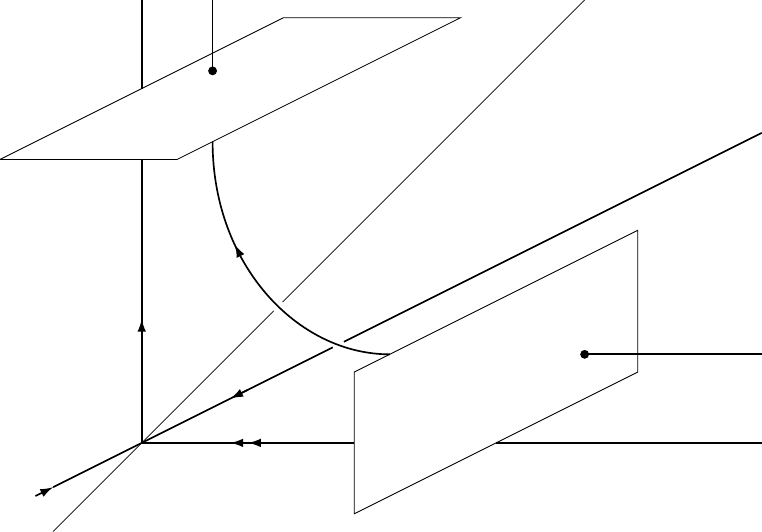}}}
% coordinate axes
\put(35,-0.5){\makebox(0,0)[rt]{$x^{ss}$}}
\put(-0.5,25){\makebox(0,0)[rt]{$x^u$}}
\put(35,17.5){\makebox(0,0)[rb]{$x^s$}}
\put(25,25){\makebox(0,0)[lt]{$x^c$}}
% in-section
\put(28.5, 8){\makebox(0,0)[l]{$\Sigma^{in}$}}
% out-section
\put(13,24.5){\makebox(0,0)[b]{$\Sigma^{out}$}}
%trajectory
\put(25,5){\makebox(0,0)[br]{$x^{in}$}}
\put(4,21){\makebox(0,0)[tl]{$x^{out}$}}
\end{picture}
\caption{\label{detail-LocalMap}Local passage $\Phi^{loc}$.}
\end{figure}

Now we come to the definition of the local In- and Out-Sections%
\begin{comment}
\begin{itemize}
\item We have to assure that the Takens-neighbourhoods are disjoint (this
is probably a standard procedure...), i.e. start with $U_{p_{k}}$
and define boxes $V_{p_{k}}$ and their union $V_{B}=\bigcup_{k=1}^{n}V_{p_{k}}$
\item also the question of ``scaling of section'' needs a comment: $x^{ss}=\epsilon$
or $x^{ss}=1$?? \end{itemize}
\end{comment}
{} illustrated in the picture above: For a point $p_{k}\in B$, we first define the 
box $V_{p}$ as $V_{p}=V_{p}(\alpha,\beta,\epsilon):=\{q=(x_{q}^{c},x_{q}^{s},x_{q}^{ss},x_{q}^{u})\in U_{p}|0\le x_{q}^{s},x_{q}^{ss},x_{q}^{u}\le\epsilon,\alpha\le x_{q}^{c}\le\beta\}$
and $\alpha,\beta,\epsilon$ are chosen so small the box $V_{p}$
lies completely inside the Takens-neighbourhood $U_{p}$. Denote their
union by $V_{B}=\bigcup_{k=1}^{n}V_{p_{k}}$.

Then define the sections by $\Sigma_{k}^{in,ss}:=V_{p_{k}}\cap\{x^{ss}=\epsilon\}$
and $\Sigma_{k}^{out}:=V_{p_{k}}\cap\{x^{u}=\epsilon\}$. Finally define the ``collections'' of sections for the whole set of
base-points $B$: $\Sigma_{B}^{in,s}=\bigcup_{k=1}^{n}\Sigma_{k}^{in,s}$
,$\Sigma_{B}^{in,ss}=\bigcup_{k=1}^{n}\Sigma_{k}^{in,ss}$ and $\Sigma_{B}^{out}=\bigcup_{k=1}^{n}\Sigma_{k}^{out}$
, and finally $\Sigma_{B}^{in}=\Sigma_{B}^{in,s}\cup\Sigma_{B}^{in,ss}$.

We need some more notation before we can introduce the main theorem of this section.
Decompose the tangent spaces of the sections defined above into the parts of the hyperbolic
direction ($V^{h}$, $W^{h}$), on the one hand, and the center-component
($V^{c}$,$W^{c}$), on the other hand. For this, let $q\in\Sigma_{B}^{in,ss}$
and $r\in\Sigma_{B}^{out}$:
\begin{itemize}
\item $T_{q}\Sigma_{B}^{in,ss}=V_{q}^{h}\oplus V_{q}^{c}$, where $V_{q}^{h}=span\{\frac{\partial}{\partial x^{s}}(q),\frac{\partial}{\partial x^{u}}(q)\}$,
i.e. the additional stable direction and the unstable direction, and
$V_{q}^{c}=span\{\frac{\partial}{\partial x^{c}}(q)\}$
\item $T_{r}\Sigma_{B}^{out}=W_{r}^{h}\oplus W_{r}^{c}$, where $W_{r}^{h}=span\{\frac{\partial}{\partial x^{ss}}(r),\frac{\partial}{\partial x^{s}}(r)\}$,
i.e. the both stable directions, because we are in the out-section,
and $W_{r}^{c}=span\{\frac{\partial}{\partial x^{c}}(r)\}$
\end{itemize}
Note that one point of this construction is to ``collect'' also
the tangent spaces like the other objects before, i.e. to talk about
the decomposition of the tangent bundle of the set $\Sigma_{B}^{out}$,
which is possible because all object depend smoothly on the base point:
\begin{itemize}
\item $T\Sigma_{B}^{out}=V^{h}\oplus V^{c}$
\item $T\Sigma_{B}^{out}=W^{h}\oplus W^{c}$
\end{itemize}
Now we are in the position to state the theorem about the local passage. Recall that $H_B$ stands for the set of all heterclinic Bianchi-II-orbits connecting base points of the set $B$ (see chapter \ref{bianchi-defs}, section \ref{kasner-map}):
\begin{thm}
Assume that, for all $p\in B$, the vector field has been according
brought to the form as in the conclusion of Theorem \ref{thm:local-form-vf}.
The local passage map $\Phi_{B}^{loc}:\Sigma_{B}^{in}\rightarrow\Sigma_{B}^{out}$
is a $C^{1}$-map that satifies, for $q\in H_B\cap\Sigma_{B}^{in}$:
\begin{itemize}
\item $\Phi_{B}^{loc}$ contracts super-linearly in the hyperbolic directions,
i.e. $d\Phi_{B}^{loc}(q)(v)=0 \; \forall v\in V_{q}^{h}$
\item $\Phi_{B}^{loc}$ is the identity in the center-direction, i.e. 
\begin{enumerate}
\item $d\Phi_{B}^{loc}(q)(V_{q}^{c})=W_{\Phi_{B}^{glob}(q)}^{c}$ 
\item $||d\Phi_{B}^{loc}(q)(v)||_{g_{p}}=||v||_{g_{p}}\forall v\in V_{q}^{c}$
\end{enumerate}
\end{itemize}
\end{thm}

\begin{proof}
Let $p\in B$ be a point from the set of admissible base points. Because
of Theorem \ref{thm:local-form-vf}, the local passage near the Kasner
circle $\Phi_{p}^{loc}$ in a neighbourhood $U_{p}$ can be calculated
explicitly (with $x_{in}^{ss}=1$ in $\Sigma_{p}^{in}$ and $x_{out}^u=1$
in $\Sigma_{p}^{out}$ after approriate scaling): 

\begin{eqnarray}
\label{compute-LP-start}
x_{out}^{s} & = & e^{\lambda_{s}t_{loc}}\cdot x_{in}^{s}=(x_{in}^{u})^{-\frac{\lambda_{s}}{\lambda_{u}}}\cdot x_{in}^{s}\\
x_{out}^{ss} & = & e^{\lambda_{ss}t_{loc}}\cdot x_{in}^{ss}=(x_{in}^{u})^{-\frac{\lambda_{ss}}{\lambda_{u}}}\\
x_{in}^{u} & = & e^{-\lambda_{u}t_{loc}}\cdot x_{out}^{u}
\label{compute-LP-end}
\end{eqnarray}

By solving the third equation for the local passage time $t_{loc}$,
one obtains the following formulas for $\Phi_{p}^{loc}:\Sigma_{p}^{in}\rightarrow\Sigma_{p}^{out}$
(when $x^{u}>0$):

\[
\Phi_{p}^{loc}(x^{c},x^{s},1,x^{u})=(x^{c},(x_{in}^{u})^{-\frac{\lambda_{s}}{\lambda_{u}}}\cdot x_{in}^{s},(x_{in}^{u})^{-\frac{\lambda_{ss}}{\lambda_{u}}},1)
\]
 and for $x^{u}=0$, we get (when following the heteroclinic orbit)

\[
\Phi_{p}^{loc}(x^{c},x^{s},1,0)=(x^{c},0,0,1)
\]

As the above equations show, the main point for understanding the
local passage is the relation of the eigenvalues. In Bianchi IX, we know
that it holds (away from the Taub points):

\[
|\lambda_{u}|<|\lambda_{s}|<|\lambda_{ss}|
\]

, i.e. the absolute value of the unstable eigenvalue is strictly smaller
than the absolute value of the two stable eigenvalues. This can be
seen from the formulas (\ref{eq:EVinBIXinU}) expressing the eigenvalues in terms of the Kasner parameter $u$, see chapter \ref{bianchi-defs}, section \ref{ev-formula-section}). That's why it holds for the fractions which appear in the exponents of the formulas above:

\[
-\frac{\lambda_{s}}{\lambda_{u}},-\frac{\lambda_{ss}}{\lambda_{u}}>1
\]

and observe that both are necessarily positive because stable and
unstable eigenvalues have opposite signs (note that this is even indepent of the chosen
time direction towards/away from the big bang). This yields the claimed
$C^{1}$-map and the super-linear contraction in the hyperbolic directions
for the map $\Phi_{p}^{loc}$.

As the vector field is completely linear in the Takens-neighbourhood,
it trivially holds that $x_{out}^{c}=e^{0}\cdot x_{in}^{c}$, i.e.
we have not drift and $\Phi^{loc}$ is just the identity in the center-direction.

These observations hold for the local passage $\Phi_{p}^{loc}:\Sigma_{p}^{in}\rightarrow\Sigma_{p}^{out}$
at any admissible base point $p\in B$, and therfor also for the collection
$\Phi_{B}^{loc}:\Sigma_{B}^{in}\rightarrow\Sigma_{B}^{out}$.
\end{proof}

\subsection{Global Passage}

Now we deal with the global passage. For the proof of the main theorem
in this section, we consider two maps which map from the respective sections onto the Kasner circle by following the heteroclinic orbit (compare \cite{beg10}, p.19):
%(for details see beguin-paper

$\alpha: H_B\cap\Sigma_{B}^{out}:\rightarrow\mathcal{K}\cap V_{B}$

$\omega: H_B\cap\Sigma_{B}^{in}:\rightarrow\mathcal{K}\cap V_{B}$
%,

where we recall that $H_B$ stands for the Bianchi-II-heteroclinics and $V_{B}$
is the collection of Takens-neighbourhoods (or the boxes, more precisely)
constructed above when dealing with the local passage. At this point,
we recall how we defined our global metric $g_B$, see (\ref{global-metric}).
It is composed of the Riemanian metric $g_{p}=(dx^{c})^{2}+(dx^{s})^{2}+(dx^{ss})^{2}+(dx^{u})^{2}$
for the Takes-coordinates in a neighbourhood $U_{p}$ near a point
$p \in B$ of the Kasner circle, and the Riemanian metric $h=dN_{1}^{2}+dN_{2}^{2}+dN_{3}^{2}+d\Sigma_{+}^{2}+d\Sigma_{-}^{2}$ otherwise. We may assume that both metics coincide when restricted to $\mathcal{K}\cap U_{p_{i}}$, because the local vector field has no center-component at the Kasner circle, i.e. one can replace center coordinate $x$
by $\phi(x)$ for a diffeo $\phi$ without changing the vector field.

This means that both maps $\alpha$, $\omega$ are local $C^{1}$-isometries for the metrics
induced by the global metric $g_{B}$ on the sets above, and we will use this fact in our proof below. 
\begin{thm}
There exits a neighbourhood $\mathcal{V}$ of $H_B\cap\Sigma_{B}^{out}$
in \textup{$\Sigma_{B}^{out}$ }\textup{\emph{such that the global
passage map}} 

\begin{eqnarray*}
\Phi_{B}^{glob}:\Sigma_{B}^{out} & \rightarrow & \Sigma_{B}^{in}\\
\mathcal{V} & \rightarrow & \Phi_{B}^{glob}(\mathcal{V})
\end{eqnarray*}

is a $C^{1}$-map on $\mathcal{V}$ and a diffeomorphism onto its image.

$\Phi_{B}^{glob}$expands in the center direction, i.e. for \textup{$r\in H_B\cap\Sigma_{B}^{out}$}, it satisfies
\begin{enumerate}
\item $d\Phi_{B}^{glob}(r)(W_{q}^{c})=V_{\Phi_{B}^{glob}(r)}^{c}$ 
%(i.e.: maps the tangent spaces in the right way)
\item $\exists\kappa>1:||d\Phi_{B}^{glob}(r)(w)||_{g_{B}}\ge\kappa||w||_{g_{B}}\forall w\in W_{r}^{c}$

\end{enumerate}
\end{thm}

\begin{proof}

\begin{comment}
include proof on $C^{1}$-depence of passage time from Haerterich-DynamikII-Skript,
Ende 1.Kapitel here, using IFT. Auch: verarbeite Lip-Abschaetzungen
der Global Passages aus Liebscher-Papers (old paper: explicit calculations,
new magneto-paper: more on differential dependence)
\end{comment}

\begin{comment}
\begin{proof}
NOTE: Béguin claims nothing for the hyperbolic directions, so there
is nothing to show :). He will use the chain rule later to show that
$d\Phi_{B}^{return}(v)=d\Phi_{B}^{glob}(\Phi_{B}^{loc})\circ d\Phi_{B}^{loc}(v)$
is already zero because the local passage is a super-linear contraction
in the hyperbolic directions. However, i would like to deal a little
bit more with the global passage, if only to prepare for Binachi B
:)\end{proof}
\end{comment}
We know that for Ordinary Differential Equation with differentiable
($C^{k}-)$vector field, there is a differentiable ($C^{k}-)$dependence
of the solution on the initial conditions (see e.g. \cite{ode-reference}). This means that in general, for any ``time-t-map''
of a differentiable flow, for fixed $t=t^{*}$ and an open subset
$U\subset\mathbb{R}^{n}$ of the phase space, we get a diffeomorphism
onto its image:

\begin{eqnarray*}
\phi_{t^{*}}:\mathbb{R}^{n} & \rightarrow & \mathbb{R}^{n}\\
U & \rightarrow & \phi_{t^{*}}(U)
\end{eqnarray*}
The Wainwright-Hsu vector field $X^{W}$ is polynomial, hence analytic,
that's why its flow $\phi_{t}(x_{0})$ does depend in a differential
(and even analytic) way on the inital condition. This means that the
map $\Phi_{p}^{glob}:\Sigma_{p}^{out}\rightarrow\Sigma_{f(p)}^{in}$
is a $C^{1}$-map and a diffeomorphism onto its image, as claimed
for the hyperbolic directions. We are left to show the second part
of the theorem, dealing with the center directions. Now let $q\in H_B\cap\Sigma_{B}^{out}$.
Then we observe that $\omega(\Phi^{glob}(q)=\omega(q)=f(\alpha(q)$,
where $f$ stands for the Kasner map. Because we have shown that both
$\alpha$ and $\omega$ are local $C^{1}$-isometries w.r.t. $g_{B}$,
we are left to prove that

\[
\exists\kappa>1:\forall p\in B,\forall v\in T_{p}\mathcal{K}:||df(p)(v)||_{g}\ge\kappa\cdot|v|_{g}
\]

, which follows directly from the definition of the Kasner map, as we consider a periodic chain which clearly keeps a minimal distance from the Taub points, where $f$ is not expanding. 

These observations hold for the global passage $\Phi_{p}^{glob}:\Sigma_{p}^{out}\rightarrow\Sigma_{f(p)}^{in}$
at any admissible base point $p\in B$, and therfor also for the collection
$\Phi_{B}^{loc}:\Sigma_{B}^{in}\rightarrow\Sigma_{B}^{out}$.
\end{proof}

\subsection{The Return Map and the Hyperbolic Structure}

As a consequence, we get the following result:
\begin{thm}
The return map $\Phi_{B}^{return}=\Phi_{B}^{glob}\circ\Phi_{B}^{loc}:\Sigma_{B}^{in}\rightarrow\Sigma_{B}^{in}$
is a $C^{1}$-map that satisfies, for $q\in H_B\cap\Sigma_{B}^{in}$\end{thm}
\begin{itemize}
\item $\Phi_{B}^{return}$ contracts super-linearly in the hyperbolic directions,
i.e. $d\Phi_{B}^{return}(q)(v)=0 \forall v\in V_{q}^{h}$
\item $\Phi_{B}^{return}$expands in the center direction, i.e. $\exists\kappa>1:\\||d\Phi_{B}^{return}(q)(v)||_{g_{B}}\ge\kappa||v||_{g_{B}}\forall v\in V_{q}^{c}$\end{itemize}
\begin{proof}
We recall the main idea behind our construction: We have shown that for the hyperbolic directions, the local passage is a contraction, while the global passsage is a diffeomorphism. Because of the differential dependence of a solution of an ODE on the initial conditions, the passage time for global passage near a heteroclinic orbit depends in a $C^{1}$-way on the base point considered. When approaching the attactor, it remains bounded, while the passage
time for the local passage tends to infinity. That's why the local passage dominates, and we get a contraction in the hyperbolic directions. In the center direction, the local passage is the identity in our local coordinate system, which yields the claimed expansion when combined with the global passage which expands the center direction. More formally, we use the chain rule $d\Phi_{B}^{return}(v)=d\Phi_{B}^{glob}(\Phi_{B}^{loc})\circ d\Phi_{B}^{loc}(v)$%
\begin{comment}
das noch genau mit den basispunkten geregelt kriegen
\end{comment}
{} to get the claims directly from our theorems above, for $q\in H_B\cap\Sigma_{B}^{in}$:

$d\Phi_{B}^{return}(q)(v)=0 \; \forall v\in V_{q}^{h}$

$||d\Phi_{B}^{return}(q)(v)||_{g_{B}}\ge\kappa||v||_{g_{B}} \forall v\in V_{q}^{c}$
\end{proof}
The theorem above means that our return map $\Phi_{B}^{return}$ has a
$C^1$-hyperbolic structure on the set $H_B\cap\Sigma_{B}^{in}$, i.e.
that it is a hyperbolic set.
Via Theorem \ref{thm-takens-appendix} (described below), this $C^1$-hyperbolic structure leads to a $C^1$-stable-manifold.

%, which is of dimension two as a subset of $\Sigma_{B}^{in}$. 

To make this more precise, consider a point $p \in B$ and observe that the heteroclinic orbit $H_{p,f(p)}$ intersects $\Sigma_{B}^{in}$ in exactly one point that we denote by $q$. We also note that $q \in (H_B \cap \Sigma_{B}^{in})$, i.e. it belongs to our hyperbolic set. Theorem \ref{thm-takens-appendix} yields a $C^1$-embedded 2-dimensional stable manifold $W_{\epsilon}^{s}(\Phi,q)$ in $\Sigma_{B}^{in}$.  And as we know that the orbits of the Bianchi IX flow are transversal to $\Sigma_{B}^{in}$, we obtain a 3-dimensional stable manifold for the base point $p$ on the Kasner circle as claimed (compare \cite{beg10}, p. 22).

\begin{comment}
we fix $\eta > 0$ and set

\[
W^s(p)=\bigcup_{-\eta \le t \le \eta} 
\]

, we get that $W^s(p)$

\end{comment}

In summary, we arrive at the following theorem, which is equivalent to Theorem \ref{main-thm-bix}:

\begin{thm} (Stable Manifolds for Points in $B$)
\label{main-thm-detailed}

Let $p\in B$, where $B$ is the set of base points of a periodic
heteroclinic chain that satisfies the Sternberg Non-Resonsonance-Conditions. Then there exists a three dimensional $C^1$-stable manifold $W^{s}(p)$ of initial conditions such that the corresponding vacuum Bianchi IX - solutions converge to the periodic heteroclinic chain towards the big bang.

%which depends continously on the base point in the $C^{1}$-topology.

%three dimensional $C^{1}$- stable - manifold of initial conditions such that the corresponding Bianchi IX - solutions converge to the periodic heteroclinic chain generated by $u$ towards the big bang

%of co-dimension 1 and the stable manifold $W^{s}(q)$ depends continously on the base point in the $C^{1}$-topology

\end{thm}

Combining this with Theorem \ref{thm-2per} and Definition \ref{adm-2per} on the admissibility of 2-periodic continued fraction developments leads immediately to Theorem \ref{2nd-main-thm-bix}.

Untill now, we have only dealt with periodic heteroclinic chains, as this was the "missing case" in the paper by Béguin, who was treating aperiodic chains. When we combine the two results, we can get 
$C^1$-stable manifolds for any points $p \in \mathcal{K}$ that do not contain "forbidden" base points in the closure of the orbit of $p$ under the Kasner map $f$, i.e.  $\overline{\{f^{n}(p)\}}\subseteq B_{\epsilon}^{T}$. For this we define $B^T_\epsilon$ to be the set of base points that satisfies the Non-Resonsonance-Conditions in order to allow for Takens Lineraization and keeps a minimum distance of $\epsilon$ from the Taub points. This second condition is trivially fullfilled for periodic chains and necessary in order to achive uniform rates of expansion/contraction for the hyperbolic structure. The reason is that both the expansion of the Kasner map as well as the contraction of the local passage breaks down at the Taub points. 

%There, the fact that the unstable eigenvalue is smaller in absolute value than both of the stable eigenvalues do not hold 

%We don't give a full description of $B^T_\epsilon$ here.

We can also elaborate a bit about what it means that solutions of Bianchi IX converges to a heteroclinic chain towards the big bang. For example, we can show that the Hausdorff distance between the heteroclinic orbits that are part of the chain and the respective piece of the Bianchi IX-orbit tends to zero. This follows from the continuity of the flow and the properties of the stable manifold (see \cite{beg10}, p.21). Thus the limit of the analysis presented here can be formulated as in Theorem \ref{thm-limit-analysis}.

\newpage

\subsection{$C^{1}$-Stable Manifolds for  $C^{1}$-Hyperbolic Sets}

We have shown that the global return map admits a $C^{1}$-hyperbolic
structure. Béguin then uses the following Theorem (see \cite{beg10}, p.18) to prove the existence of a  $C^{1}$-stable manifold: Theorem \ref{thm-takens-appendix} shows that a $C^{1}-$Hyperbolic Structure leads
to a $C^{1}$- stable manifold, where the "index $s$" of the hyperbolic set stands for the dimension of the stable subbundle of the tangen bundle $TM$ (i.e. $s=dim (X_{p})$ in the notation of Definition \ref{def-hyp-structure}). In addition, the theorem specifies the dependence of this
manifold on the base point as well as the convergence rate:

\begin{thm}
\label{thm-takens-appendix}
Let $\Phi:M\rightarrow M$ be a $C^{1}$map on a manifold $M$, and
$C$ be a compact subset of $M$ which is a hyperbolic set of index
s for the map $\Phi$. Then, for every $\epsilon$ small enough, for
every $q\in C$, the set

\[
W_{\epsilon}^{s}(\Phi,q):=\{r\in M|dist(\Phi^{n}(r),\Phi^{n}(q))\leq\epsilon\, for\, every\, n\geq0\}
\]

is a $C^{1}$embedded s-dimensional disc, tangent to $F_{q}^{s}$
at q, depending continuously on q (for the $C^{1}$topology on the
space of embeddings). Moreover, if $\mu$ is a contraction rate for
$\Phi$ on $C$, then there exists a constant $\kappa$ such that,
for every $\epsilon$ small enough, for every $q\in C$ and every
$r\in W_{\epsilon}^{s}(\Phi,q)$

\[
dist_{g}(\Phi^{n}(r),\Phi^{n}(q))\leq\kappa\mu^{n}
\]

\end{thm}

Béguin names the book \cite{takens93} by Palis and Takens (page 167) as a reference for Theorem \ref{thm-takens-appendix}. In this section of the Appendix
``Hyperbolicity: Stable Manifolds and Foliations'', the authors 
deal with hyperbolic sets for endomorphisms, but results are only
sketched and no proofs included. However, there are classic sources for stable
manifold theorems of hyperbolic sets: Partly based on an earlier paper
(\cite{hpugh68}), Hirsch and Pugh prove such a theorem in \cite{hpugh70},
which is a chapter of the book ``Global Analysis'' collecting
the proceedings a symposium held on the topic in Berkeley, California,
in 1968, and seems to be the first time such a result is proved.
We will introduce the theorem by Hirsch/Pugh below, it can be used instead of \ref{thm-takens-appendix} in order to prove our Theorem \ref{main-thm-detailed}.

% and also comment on how it is proved.

%\newpage

\subsection{Generalized Stable Manifold Theorem by Hirsch/Pugh}

\begin{thm*}
(Generalized Stable Manifold Theorem) Let $U$ be an open set in a
smooth manifold $M(dim<\infty)$ and $f:U\rightarrow M$ a $C^{1}$-
map. Let $\Lambda\subset U$ be a compact hyperbolic set and call
the invariant splitting $T_{\Lambda}M=E_{1}\oplus E_{2}$. Then there
is a neighbourhood $V$ of $\Lambda$, and submanifolds $W^{s}(x),W^{u}(x)$
tangent to $E_{2}(x)$ and $E_{1}(x)$ respectively for each $x\in\Lambda$ such
that

\[
W^{s}(x)=\{y\in V|\lim_{n\rightarrow\infty}d(f\restriction V)^{n}y,f\restriction V)^{n}x)=0\}
\]

If $f$ is $C^{k}$, so is $W^{s}(x),$ and it depends continously
on $f$ in the $C^{k}$-topology. Moreover, $W^{s}(x)$ and its derivatives
along $W^{s}(x)$ up to order $k$ depend continously on $x$. In
addition, there exist numbers $K>0,\lambda<1$ \textup{such that if
$x\in\Lambda,z\in W_{x}$ and $n\in\mathbb{Z}_{+}$ then }the following
holds:

\[
d(f^{n}(x),f^{n}(z)\le K\lambda^{n}
\]
\begin{comment}
letzter punkt kommt von thm 3.2, p.149 aus buch (beim beweis steht
nur ``folgt aus analogem fakt für stablemf von fixed points -> mal
selbst checken!!!)
\end{comment}

\end{thm*}

In \cite{hpugh70}, the proof of the generalized stable manifold theorem is outlined as follows:
%proceeds in the following steps:

\begin{enumerate}
\item Let $E=E_{1}\times E_{2}$ be a Banach space; $T:E\rightarrow E$
a hyperbolic linear map expanding along $E_{1}$ and contracting along
$E_{2}$; $E(r)\subset E$ the ball of radius r, and $f:E(r)\rightarrow E$
a Lipschitz pertubation of $T\restriction E(r)$. The unstable manifold
$W$ for $f$ will be the graph of a map $g:E_{1}(r)\rightarrow E_{2}(r)$
which satisfies $W=f(W)\cap E(r)$. Then the following map
$\Gamma_{f}$ is considered (in a suitable function space $G$ of
maps $g$):\\
\\
 $graph[\Gamma_{f}(g)]=E(r)\cap f(graph[g])$\\
\\
i.e. $\Gamma_{f}$ is the graph transform of $g$ by $f$.
The fixed point $g_{0}$ of $\Gamma_{f}$ gives the unstable manifold
of $f$ - its existence is proved by the contracting map principle
if $f$ is sufficiently close to $T$ pointwise, and the Lipschitz
constant of $f-T$ is small enough.
\item If $f$ is $C^{k}$ so is $g_{0}$, which is proved by induction on
$k$. The successive approximations $\Gamma_{f}^{n}(g)$ converge
$C^{k}$ to $g_{0}$ - here the Fibre Contraction Theorem is used.
\item Let $\Gamma\subset U$ be a hyperbolic set. Let $\mathcal{M}$ be
the Banach manifold of bounded maps $\Lambda\rightarrow M$, and $i\in\mathcal{M}$
the inclusion of $\Lambda$. Let $\mathcal{U}=\{h\in\mathcal{M}|h(\Lambda)\subset U\}$.
Define $f_{*}:\mathcal{U}\rightarrow\mathcal{M}$ by\\
\\
$f_{*}(h)=f\circ h\circ f^{-1}$\\
\\
Then $f_{*}$has a hyperbolic fixed point at $i$. By the first point,
$f_{*}$ hast a stable manifold $\mathfrak{\mathcal{W}}^{s}\subset\mathcal{M}$.
For each $x\in\mathcal{M}$, define $W^{s}(x)=ev_{x}(\mathcal{W}^{s})=\{y\in M|y=\gamma(x)\textrm{ for some }\gamma\in\mathcal{W}^{s}\}$. This yields a system of stable manifolds for $f$ along
$\Lambda$
\end{enumerate}

Point (1) of the outline above involves a graph-transform of Lipschitz-graphs (see e.g. \cite{rob95}, and compare also \cite{lieetal10, lieetal12}, where it is described in detail how a graph transform can be used to prove Lipschitz-stable-manifolds in Bianchi models even without linearizing at the Kasner circle). Point (3) reduces the proof of a stable manifold for a hyperbolic set to the case of a fixed point, in a suitable chosen infinite-dimensional space (compare also \cite{takens93}, p.157).

\subsection{Differentiability of the Stable Manifold}
\label{diff-stable-mf}

In step (2) above, the differentiability of the stable manifold is proved by the Fibre Contraction Principle (see \cite{hpugh70}, p.136 or \cite{hpugh-shub-book}, p.25). As the differentiability of the stable manifold is the main point of our Theorem \ref{main-thm-bix}, we will comment a bit how this is done.
For the invariant section (which will be the desired stable manifold) to be differentiable,
it is not enough to obtain a fibre contraction. One important point is that it may not contract
more along the base space than along the fibres (compare  \cite{hpugh-shub-book}, p.26), otherwise
there are examples where there is no differentiable invariant section
(see e.g. \cite{rob95}, p. 435). That's why we need additional conditions
that assure that the contraction on fibres is stronger than the contraction
in the base space to prove a ``$C^{r}$ section Theorem'' (\cite{rob95},
p. 436).

An alternative approach is the method of cones (e.g. taken by Robinson \cite{rob95}, p.185).
As above, a stable manifold that is only Lipschitz is obtained in a first step, and then it is shown that the
obtained manifold is infact $C^{k}$ if the original map has this smoothness property (\cite{rob95}, p.194). 

Finally, the book \cite{shub-book} also contains stable manifold theorems both for fixed points (chapter 5) and hyperbolic sets (chapter 6), in an abstract setting similar to \cite{hpugh70}, and also deals with the differentiability question (see \cite{shub-book}, p.39).

\newpage

\chapter*{Discussion and Outlook}
\label{conl-outl}

We have shown that there are periodic heteroclinic chains in Bianchi IX for which there exisist
$C^{1}$- Stable - Manifolds of orbits that follow these chains towards the big bang. This result is new, and should be compared with the two existing rigorous results on stable manifolds for orbits of the
Kasner map in Bianchi IX: B\'eguin showed the existence of $C^{1}$- stable- manifolds for aperiodic orbits of the Kasner map (\cite{beg10}), while Liebscher and co-authors (\cite{lieetal10,lieetal12}) showed the existence of Lipschitz-stable-manifolds for arbitrary orbits of the Kasner map not accumulating at one of the Taub points (B\'eguin also had to demand the latter condition).

Our result significantly extends B\'eguins results, who had to exclude all orbits that are perioidic or accumulate on any periodic orbit, a limitation which we were able to overcome. The techniques by Liebscher et al are able to treat both periodic and aperiodic chains, but yielded only Lipschitz-manifolds, i.e. the leaves of the foliation have less regularity. 

But be aware that even though the stable manifolds constructed by B\'eguin and ourselves are $C^1$, this concerns only the regularity of the leaves of the foliation, and not the dependence on the base point. We do not get a $C^1$-foliation which would mean a $C^1$-dependece on the base point, but only a  $C^0$-dependence of the ($C^1$-)leaves in the  $C^1$-topology.

These aspects play a crucuial role when discussing the genericity of the foliation-results in BIX, i.e. how generic the set of initial conditions is both "down on the Kasner circle", as well as in the full space of trajectories. This involves delicate distinctions between topological vs. measure-theoretic genericity, and is subject of current research (for partial results, see \cite{reitru10}\footnote{in \cite{reitru10} it is shown that there are trajectories converging to every formal sequence given by a Kasner parameter $u$ with at most polynomially bounded continued fraction expansion. This covers a set of full measure on the Kasner circle, but this does not mean that the set of coresponding initial conditions in a neighborhood of the Kasner circle has full measure. The reason is that there are counterexamples, i.e. it is possible to construct foliations where a countable set of "leaves" is attached to a set of base points that has full measure in the base space.}).

\newpage

\appendix

%\chapter{}

\label{appendix-BIX-NRC}

\section{Symbolic Computations with Mathematica} \label{appendixBIX}

%Below you can see the output of the computations described in section \ref{Details-Section-Formulas} by Mathematica.
%\footnote{The original output from the Mathematica-files, which can be downloaded from \cite{joe-bianchi-homepage13}, was slightly formatted in some cases for aesthetic reasons}

\subsection{Constant Continued Fraction Expansion}  u=[a,a,...]\\

For u={[}m,a,a,...{]} and m=1...a, AND a= 1

m= 1 alpha= 16 beta= 4 k1= -1 k2= 1 k3= -1\\

For u={[}m,a,a,...{]} and m=1...a, AND a= 2

m= 1 alpha= 12 beta= 3 k1= 1 k2= 2 k3= -1

m= 2 alpha= 24 beta= 5 k1= -2 k2= 1 k3= -1\\

For u={[}m,a,a,...{]} and m=1...a, AND a= 3

m= 1 alpha= 11 beta= 3 k1= 3 k2= 3 k3= -1

m= 2 alpha= 19 beta= 4 k1= 1 k2= 3 k3= -1

m= 3 alpha= 33 beta= 6 k1= -3 k2= 1 k3= -1\\

For u={[}m,a,a,...{]} and m=1...a, AND a= 4

m= 1 alpha= 11 beta= 3 k1= 5 k2= 4 k3= -1

m= 2 alpha= 18 beta= 4 k1= 4 k2= 5 k3= -1

m= 3 alpha= 28 beta= 5 k1= 1 k2= 4 k3= -1

m= 4 alpha= 45 beta= 7 k1= -4 k2= 1 k3= -1\\

For u={[}m,a,a,...{]} and m=1...a, AND a= 5

m= 1 alpha= 11 beta= 3 k1= 7 k2= 5 k3= -1

m= 2 alpha= 18 beta= 4 k1= 7 k2= 7 k3= -1

m= 3 alpha= 27 beta= 5 k1= 5 k2= 7 k3= -1

m= 4 alpha= 39 beta= 6 k1= 1 k2= 5 k3= -1

m= 5 alpha= 59 beta= 8 k1= -5 k2= 1 k3= -1\\

For u={[}m,a,a,...{]} and m=1...a, AND a= 6

m= 1 alpha= 11 beta= 3 k1= 9 k2= 6 k3= -1

m= 2 alpha= 18 beta= 4 k1= 10 k2= 9 k3= -1

m= 3 alpha= 27 beta= 5 k1= 9 k2= 10 k3= -1

m= 4 alpha= 38 beta= 6 k1= 6 k2= 9 k3= -1

m= 5 alpha= 59 beta= 8 k1= 1 k2= 6 k3= -1

m= 6 alpha= 75 beta= 9 k1= -6 k2= 1 k3= -1\\

For u={[}m,a,a,...{]} and m=1...a, AND a= 7

m= 1 alpha= 11 beta= 3 k1= 11 k2= 7 k3= -1

m= 2 alpha= 18 beta= 4 k1= 13 k2= 11 k3= -1

m= 3 alpha= 27 beta= 5 k1= 13 k2= 13 k3= -1

m= 4 alpha= 38 beta= 6 k1= 11 k2= 13 k3= -1

m= 5 alpha= 51 beta= 7 k1= 7 k2= 11 k3= -1

m= 6 alpha= 67 beta= 8 k1= 1 k2= 7 k3= -1

m= 7 alpha= 93 beta= 10 k1= -7 k2= 1 k3= -1\\

For u={[}m,a,a,...{]} and m=1...a, AND a= 8

m= 1 alpha= 11 beta= 3 k1= 13 k2= 8 k3= -1

m= 2 alpha= 18 beta= 4 k1= 16 k2= 13 k3= -1

m= 3 alpha= 27 beta= 5 k1= 17 k2= 16 k3= -1

m= 4 alpha= 38 beta= 6 k1= 16 k2= 17 k3= -1

m= 5 alpha= 51 beta= 7 k1= 13 k2= 16 k3= -1

m= 6 alpha= 66 beta= 8 k1= 8 k2= 13 k3= -1

m= 7 alpha= 84 beta= 9 k1= 1 k2= 8 k3= -1

m= 8 alpha= 113 beta= 11 k1= -8 k2= 1 k3= -1\\

For u={[}m,a,a,...{]} and m=1...a, AND a= 9

m= 1 alpha= 11 beta= 3 k1= 15 k2= 9 k3= -1

m= 2 alpha= 18 beta= 4 k1= 19 k2= 15 k3= -1

m= 3 alpha= 27 beta= 5 k1= 21 k2= 19 k3= -1

m= 4 alpha= 38 beta= 6 k1= 21 k2= 21 k3= -1

m= 5 alpha= 51 beta= 7 k1= 19 k2= 21 k3= -1

m= 6 alpha= 66 beta= 8 k1= 15 k2= 19 k3= -1

m= 7 alpha= 83 beta= 9 k1= 9 k2= 15 k3= -1

m= 8 alpha= 103 beta= 10 k1= 1 k2= 9 k3= -1

m= 9 alpha= 135 beta= 12 k1= -9 k2= 1 k3= -1\\

\newpage

\subsection{2-Periodic Continued Fraction Expansion} u={[}a,b,...{]}\\
\label{appendix-2periodic}

Now use a= 2  and b= 3\\

For u=[m,a,b,a,b,...] and m=1...b

m= 1  alpha=  15  beta= 4   k1= 7   k2= 7   k3= -2

m= 2  alpha=  24  beta= 5   k1= 3   k2= 7   k3= -2

m= 3  alpha=  34  beta= 6   k1= -5   k2= 3   k3= -2\\

For u=[m,b,a,b,a,...] and m=1...a

m= 1  alpha=  11  beta= 3   k1= 2   k2= 5   k3= -3

m= 2  alpha=  19  beta= 4   k1= -7   k2= 2   k3= -3\\
 
Now use a= 3 and b= 5\\

For u={[}m,a,b,a,b,...{]} and m=1...b

m= 1 alpha= 11 beta= 3 k1= 23 k2= 17 k3= -3

m= 2 alpha= 23 beta= 5 k1= 23 k2= 23 k3= -3

m= 3 alpha= 33 beta= 6 k1= 17 k2= 23 k3= -3

m= 4 alpha= 46 beta= 7 k1= 5 k2= 17 k3= -3

m= 5 alpha= 60 beta= 8 k1= -13 k2= 5 k3= -3\\

For u={[}m,b,a,b,a,...{]} and m=1...a

m= 1 alpha= 11 beta= 3 k1= 13 k2= 13 k3= -5

m= 2 alpha= 18 beta= 4 k1= 3 k2= 13 k3= -5

m= 3 alpha= 27 beta= 5 k1= -17 k2= 3 k3= -5\\

Now use a= 1  and b= 2\\

For u={[}m,a,b,a,b,...{]} and m=1...b

m= 1  alpha=  16  beta= 4   k1= 2   k2= 3   k3= -1

m= 2  alpha=  25  beta= 5   k1= -1   k2= 2   k3= -1\\

For u={[}m,b,a,b,a,...{]} and m=1...a

m= 1  alpha=  12  beta= 3   k1= -3   k2= 1   k3= -2\\

Now use a= 2  and b= 4\\

 For u=[m,a,b,a,b,...] and m=1...b

 m= 1  alpha=  15  beta= 4   k1= 12   k2= 10   k3= -2

 m= 2  alpha=  24  beta= 5   k1= 10   k2= 12   k3= -2

 m= 3  alpha=  34  beta= 6   k1= 4   k2= 10   k3= -2

 m= 4  alpha=  47  beta= 7   k1= -6   k2= 4   k3= -2\\

 For u=[m,b,a,b,a,...] and m=1...a

 m= 1  alpha=  11  beta= 3   k1= 2   k2= 6   k3= -4

 m= 2  alpha=  18  beta= 4   k1= -10   k2= 2   k3= -4\\

\subsection{3-Periodic Continued Fraction Expansion} u={[}a,b,c,...{]}\\
\label{appendix-3periodic}

Use a= 1 and b= 1 and c=2\\

 For u={[}m,b,c,a,...{]} and m=1...a

 m= 1  alpha=  16  beta= 4   k1= 5   k2= -2   k3= 3\\

 For u={[}m,c,a,b,...{]} and m=1...b

 m= 1  alpha=  12  beta= 3   k1= 2   k2= -3   k3= 3\\

 For u={[}m,a,b,c,...{]} and m=1...c

 m= 1  alpha=  16  beta= 4   k1= -3   k2= -5   k3= 2

 m= 2  alpha=  24  beta= 5   k1= 3   k2= -3   k3= 2\\

Use a= 3 and b= 3 and c=2\\

For u={[}m,b,c,a,...{]} and m=1...a

m= 1 alpha= 11 beta= 3 k1= -23 k2= -22 k3= 7

m= 2 alpha= 19 beta= 4 k1= -10 k2= -23 k3= 7

m= 3 alpha= 33 beta= 6 k1= 17 k2= -10 k3= 7\\

For u={[}m,c,a,b,...{]} and m=1...b

m= 1 alpha= 15 beta= 4 k1= -22 k2= -23 k3= 7

m= 2 alpha= 24 beta= 5 k1= -7 k2= -22 k3= 7

m= 3 alpha= 34 beta= 6 k1= 22 k2= -7 k3= 7\\

For u={[}m,a,b,c,...{]} and m=1...c

m= 1 alpha= 11 beta= 3 k1= -7 k2= -17 k3= 10

m= 2 alpha= 23 beta= 5 k1= 23 k2= -7 k3= 10\\

Use a=b=c=1 (consistency check):\\

For u=[m,b,c,a,...] and m=1...a

m= 1  alpha=  16  beta= 4   k1= 2   k2= -2   k3= 2\\

For u=[m,c,a,b,...] and m=1...b

m= 1  alpha=  16  beta= 4   k1= 2   k2= -2   k3= 2\\

For u=[m,a,b,c,...] and m=1...c

m= 1  alpha=  16  beta= 4   k1= 2   k2= -2   k3= 2\\

Use a=b=c=3 (consistency check):\\

For u=[m,b,c,a,...] and m=1...a

m= 1  alpha=  11  beta= 3   k1= -30   k2= -30   k3= 10

m= 2  alpha=  19  beta= 4   k1= -10   k2= -30   k3= 10

m= 3  alpha=  33  beta= 6   k1= 30   k2= -10   k3= 10\\

For u=[m,c,a,b,...] and m=1...b

m= 1  alpha=  11  beta= 3   k1= -30   k2= -30   k3= 10

m= 2  alpha=  19  beta= 4   k1= -10   k2= -30   k3= 10

m= 3  alpha=  33  beta= 6   k1= 30   k2= -10   k3= 10\\

For u=[m,a,b,c,...] and m=1...c

m= 1  alpha=  11  beta= 3   k1= -30   k2= -30   k3= 10

m= 2  alpha=  19  beta= 4   k1= -10   k2= -30   k3= 10

m= 3  alpha=  33  beta= 6   k1= 30   k2= -10   k3= 10\\

\subsection{Pre-Periodic Sequences} u=[m,b,a,b,a,...] and u=[m,a,b,c,...]\\

Now use a= 3, b= 2  and M= 5 for u=[m,b,a,b,a,...], m=1...M\\

m= 1  alpha=  15  beta= 4   k1= 7   k2= 7   k3= -2

m= 2  alpha=  24  beta= 5   k1= 3   k2= 7   k3= -2

m= 3  alpha=  34  beta= 6   k1= -5   k2= 3   k3= -2

m= 4  alpha=  47  beta= 7   k1= -17   k2= -5   k3= -2

m= 5  alpha=  61  beta= 8   k1= -33   k2= -17   k3= -2\\

Now use a= 1, b= 1  and c= 2  and M= 3 for u=[m,a,b,c,...], m=1...M\\

m= 1  alpha=  16  beta= 4   k1= -3   k2= -5   k3= 2

m= 2  alpha=  24  beta= 5   k1= 3   k2= -3   k3= 2

m= 3  alpha=  35  beta= 6   k1= 13   k2= 3   k3= 2

\newpage

We can summarize our results from Mathematica as follows, where the smoothness of the coordinate change is set to one ($k=1$ in the Takens Theorem):
\begin{itemize}
\item for constant contiued fraction expansions the conditions are violated
in the cases $u=[m,a,a,...]$ for $a=1...9$, so there is no simple
infinite periodic heteroclinic chain with constant continued fraction
development. We see, for example, in the case $u=[1,1,...]$ of the 3-cycle,
it holds that $(k_{1},k_{2},k_{3})=(-1,1,-1)$, which means that $\lambda_{2}=\lambda_{1}+\lambda_{3}$,
which can be checked directly and serves as a consistency check.
\item for 2-periodic continued fractions like $u=[2,3,2,3,...]$ or $u=[3,5,3,5,...]$, the Resonance Sign Condition (RSC) is violated, i.e. Takens-Linearization is possible. But note that for this argument to work, we have to require the coefficients to be greater than one, even after cancelling out a possible common factor. This is illustrated  by the examples $u=[1,2,1,2,...]$ and $u=[2,4,2,4,...]$.
\item For $u={[}1,1,2,1,1,2,...{]}$, the RSC is also violated, illustrating the fact that we don't have to require the coefficients to be greater than one if the period is greater than two. For $u={[}3,3,2,3,3,2,...{]}$, the RSC is also violated. However, even without using this fact, the chain would qualify for Takens Linearization, as the sum of the order of the resonances is always greater than the required $\alpha$ at all base points.

\item We have also included the examples for $u=[a,b,c,a,b,c,...]$ with $a=b=c=1$ and $a=b=c=3$ as consistency check: the formulas remain correct, but due to a common factor in the resulting coefficients, there is an earlier resonance that we already found in the section on constant contined fraction expansions.

\item the 1-pre-periodic sequences $u=[3,1,1,2,1,1,2,...]$ and $u=[5,3,2,3,2,...]$ show that if the first coefficient $m$ is bigger than the ones that follow, it cannot be assured that the NRC's are met: In the first case, this fails for $m=3$, in the second case for $m=4$ and $m=5$, which means that Takens Linearization is not possible.

\end{itemize}


\newpage

\begin{thebibliography}{99}

%thesen literaturverzeichnis:

% stuff i want to cite but need to find explicit reference
% diff-dependence of solution of ode on initial conditions: e.g. Hartmann o.ä.

% TODO Literaturverzeichnis:
% erstmal inhalt komplett abschließen uns sehen, was ich zitiere, dann genau das formatieren!!!
% als letztes alphabetisch sortieren -> das krieg ich auch als letzten schritt nachts hin, aber dauert :)

%% model format: 

% Autor in \bibitem-Zeile, Punkt dahinter
% und einfach ein "{\em }"=emphasize um den Titel und gut, is eh scheiß egal! 

%\bibitem{waiell97} J. Wainwright and G.F.R. Ellis.
%\newblock {\em Dynamical systems in cosmology}.
%\newblock (Cambridge University Press, Cambridge, 1997).

% stuff i cite, in the right format

\bibitem{ode-reference} H. Amann.
\newblock {\em Ordinary differential equations. }
\newblock Walter de Gruyter, 1990.

%B

\bibitem{beg10} F.~B\'eguin.
\newblock {\em Aperiodic oscillatory asymptotic behavior for some Bianchi
spacetimes. }
\newblock Class.\ Quantum\ Grav.\ 27, 2010.

\bibitem{bkl70} V.A.~Belinski\v{\i}, I.M.~Khalatnikov, and E.M.~Lifshitz.
\newblock {\em Oscillatory approach to a singular point in the relativistic cosmology. }
\newblock Adv.~Phys. 19, 1970.

\bibitem{bkl82} V.A.~Belinski\v{\i}, I.M.~Khalatnikov, and E.M.~Lifshitz.
\newblock {\em A general solution of the Einstein equations with a time singularity. }
\newblock Adv.~Phys.~ 31, 1982.

\bibitem{bianchi-orginal} L. Bianchi.
\newblock {\em Sugli spazii a tre dimensioni che ammettono un gruppo continuo di movimenti. (On the spaces of three dimensions that admit a continuous group of movements). }
\newblock Soc. Ital. Sci. Mem. di Mat. 11, 267, 1898.

\bibitem{bron94}  I. U. Bronstein, A. Y. Kopanskii. 
\newblock {\em Smooth Invariant Manifolds and Normal Forms. }
\newblock World Scientific, 1994.

\bibitem{buchner-evolution} J. Buchner.
\newblock {\em The simplest form of Evolution Equations containing both Gowdy and the exceptional Bianchi cosmological models. }
\newblock http://dynamics.mi.fu-berlin.de/preprints/buchner-g2-equations.pdf, 2010.

% E

\bibitem{einstein15a} A. Einstein. 
\newblock {\em Die Feldgleichungen der Gravitation. }
\newblock Sitzungsberichte der Königlich Preußischen Akademie der Wissenschaften (Berlin), p. 844-847, 1915.

\bibitem{einstein15b} A. Einstein.
\newblock {\em  Zur allgemeinen Relativitätstheorie. }
\newblock Sitzungsberichte der Königlich Preußischen Akademie der Wissenschaften (Berlin), p. 778-786, 1915.

\bibitem{einstein16} A. Einstein.
\newblock {\em Die Grundlage der Allgemeinen Relativitätstheorie. }
\newblock Annalen der Physik, Volume 354, Issue 7, p. 769–822, 1916.

\bibitem{elsetal02} H.~van~Elst, C.~Uggla, and J.~Wainwright.
\newblock {\em Dynamical systems approach to G2 cosmology. }
\newblock Class.\ Quantum Grav. 19, 2002. 


% G

\bibitem{lafo04} S. Gallot, D. Hulin, J. Lafontaine.
\newblock {\em Riemannian Geometry. }
\newblock Springer, 3rd edition, 2004.

\bibitem{grob59} D. M. Grobman.
\newblock {\em Homeomorphism of systems of differential equations. }
\newblock Doklady Akad. Nauk SSSR 128, p. 880–881, 1959.

\bibitem{grob62} D. M. Grobman.
\newblock {\em Topological classification of neighborhoods of a singularity in n-space. }
\newblock Mat. Sb. (N.S.), 56(98):1, p. 77–94, 1962.

% H

\bibitem{hart60b} P. Hartman.
\newblock {\em A lemma in the theory of structural stability of differential equations. }
\newblock Proc. A.M.S. 11 (4): 610–620, 1960.

\bibitem{hartman60a} P. Hartman.
\newblock {\em On local homeomorphisms of Euclidean spaces. }
\newblock Bol. Soc. Mat. Mexicana 5, 220-241, 1960.

\bibitem{heiugg09b} J.~M.~Heinzle and C.~Uggla.
\newblock {\em Mixmaster: Fact and Belief. }
\newblock Class.\ Quantum\ Grav.\ 26, 2009.

\bibitem{heietal09} J.M. Heinzle, C. Uggla, and N. R\"ohr.
\newblock {\em The cosmological billiard attractor. }
\newblock Adv.\ Theor.\ Math.\ Phys. 13, 2009.

\begin{comment}
\bibitem{heiugg09b} J.~M.~Heinzle and C.~Uggla.
\newblock {\em A new proof of the Bianchi type IX attractor theorem. }
\newblock Class.\ Quantum\ Grav.\ 26, 2009.

\bibitem{heiugg10} J.~M. Heinzle and C. Uggla.
\newblock {\em Monotonic functions in Bianchi models: why they exist and how to find them. }
\newblock Class.\ Quantum\ Grav.\ 27, 2010.
\end{comment}

\bibitem{heietal12} J.~M.~Heinzle, C.~Uggla, W.~C.~Lim.
\newblock {\em Spike Oscillations. }
\newblock Phys.\ Rev.\ D 86, 2012.

\bibitem{heiugg13} J.~M. Heinzle and C. Uggla.
\newblock {\em Spike statistics. }
\newblock Gen.\ Rel.\ Grav.\ 45, 2013. %939-957

\bibitem{hew03} C. G. Hewitt, J.T. Horwood, J. Wainwright.
\newblock {\em Asymptotic Dynamics of the Exceptional Bianchi Cosmologies. }
\newblock Classical and Quantum Gravity, 20, p. 1743-56, 2003.

\bibitem{hpugh68} Morris W. Hirsch and Charles C. Pugh.
\newblock {\em Stable Manifolds for Hyperbolic Sets. }
\newblock Bull. Amer. Math. Soc. Volume 75, Number 1, p. 149-152, 1969.

\bibitem{hpugh70} Morris W. Hirsch and Charles C. Pugh.
\newblock {\em Stable Manifolds and Hyperbolic Sets. }
\newblock In: Global Analysis, Proceedings of the Symposium, vol. 14, pp. 133-163, AMS, Providende, RI, 1970.

\bibitem{hpugh-shub-book} M. W. Hirsch, C. C. Pugh, and M. Shub.
\newblock {\em Invariant manifolds. }
\newblock Lecture Notes in Mathematics, Vol. 583, Springer-Verlag, Berlin, 1977.

% K

\bibitem{khin49} A. Khintchine.
\newblock {\em Kettenbrüche. }
\newblock B. G. Teubner, Leipzig, 1956 (2nd edition).

% L

\bibitem{liebscher-habil} S. Liebscher.
\newblock {\em Bifurcation without parameters. }
\newblock Habilitationsschrift, Freie Universitiät Berlin, 2012. 

\bibitem{lieetal10} S.~Liebscher, J.~H\"arterich, K.~Webster, and M.~Georgi.
\newblock {\em Ancient Dynamics in Bianchi Models: Approach to Periodic Cycles. }
\newblock Commun.\ Math.\ Phys.\ 305, 2011.

\bibitem{lieetal12} S.~Liebscher, A.~D.~Rendall, and S.~B.~Tchapnda.
\newblock {\em Oscillatory singularities in Bianchi models with magnetic fields. }
\newblock arXiv:1207.2655, 2012.

\bibitem{lk63} E.M.~Lifshitz and I.M.~Khalatnikov.
\newblock {\em Investigations in relativistic cosmology. }
\newblock Adv.~Phys. 12, 1963.

\bibitem{lim04} W.C.~Lim.
\newblock {\em The Dynamics of Inhomogeneous Cosmologies. }
\newblock Ph.~D. thesis, University of Waterloo, 2004; arXiv:gr-qc/0410126.

\bibitem{lim08} W.C.~Lim.
\newblock {\em New explicit spike solution -- non-local component of the generalized Mixmaster  attractor. }
\newblock Class. Quantum Grav. 25, 2008.

\bibitem{limetal09} W.C.~Lim, L.~Andersson, D.~Garfinkle and F.~Pretorius.
\newblock {\em Spikes in the Mixmaster regime of $G_2$ cosmologies. }
\newblock Phys. Rev. D 79, 2009.

% M 

\bibitem{matlab-ode} Matlab Documentation Center.\\
\newblock {\em Numerical Integration and Differential Equations. }
\newblock http://www.mathworks.com/help/matlab/ref/ode113.html, 2013.

\bibitem{mis69a} C.~W.~Misner.
\newblock {\em Mixmaster universe. }
\newblock Phys.\ Rev.\ Lett. 22, 1969.

\bibitem{wheeler73} C.W. Misner, K.S. Thorne, and J.A. Wheeler.
\newblock {\em Gravitation}.
\newblock W.H. Freeman and Company, San Francisco, 1973.


% P


\bibitem{takens93} J. Palis, F. Takens. 
\newblock {\em Hyperbolicity and sensitive chaotic dynamics at homoclinic bifurcations. }
\newblock Cambridge University Press, 1993.

\bibitem{per54} O. Perron.
\newblock {\em Die Lehre von den Kettenbrüchen.  }
\newblock B. G. Teubner, Leipzig 1954 (3rd edition).

% R

\bibitem{reitru10} M.~Reiterer and E.~Trubowitz.
\newblock {\em The BKL Conjectures for Spatially Homogeneous
Spacetimes. }
\newblock arXiv:1005.4908v2, 2010.

\bibitem{ren97} A.~D.~Rendall.
\newblock {\em Global dynamics of the mixmaster model. }
\newblock Class. Quantum Grav. 14, 1997. 

\bibitem{ren05} A. D. Rendall.
\newblock {\em The nature of spacetime singularities. }
\newblock In: 100 years of relativity, p. 76 - 92, World Scientific, 2005.

\bibitem{ren08} A.D.~Rendall.
\newblock {\em Partial Differential Equation in General Relativity. }
\newblock Oxford University Press, Oxford, 2008.

\bibitem{rin01} H. Ringstr\"om.
\newblock {\em The Bianchi IX attractor. }
\newblock Annales Henri Poincar\'e 2, 2001.

\bibitem{rin09} H.~Ringstr\"om.
\newblock {\em The Cauchy Problem in General Relativity. }
\newblock ESI Lectures in Mathematics and Physics, 2009.

\bibitem{rob95} C. Robinson. 
\newblock {\em Dynamical Systems: Stability, Symbolic Dynamics, and Chaos. }
\newblock CRC Press, 1995.

%S

\bibitem{sell85} G R Sell. 
\newblock {\em Obstacles to Linearization. }
\newblock Differential Equations, Vol 20, pages 341-345, 1985.

\bibitem{SSTC98} L.~P. Shilnikov, A.~L. Shilnikov, D.~V. Turaev, and L.~O. Chua.
\newblock {\em Methods of Qualitative Theory in Nonlinear Dynamics I.}
\newblock  Volume~4 of Series on Nonlinear Science, Series A, World Scientific, 1998.

\bibitem{shosh72} A. Shoshitaishvili.
\newblock {\em Bifurcations of topological type at singular points of parametrized vector fields. }
\newblock Func anal Appl 6:169-170, 1972.

\bibitem{shosh75} A. Shoshitaishvili.
\newblock {\em Bifurcations of topological type of a vector field near a singular point. }
\newblock Trudy Petrovsky seminar, vol. 1, Moscow University Press, Moscow, pp. 279-309, 1975.

\bibitem{shub-book} M.Shub.
\newblock {\em Global Stability of Dynamical Systems. }
\newblock Springer, 1986.

\bibitem{spivak99} M. Spivak.
\newblock {\em A Comprehensive Introduction to Differential Geometry. }
\newblock Publish or Perish; 3rd edition, 1999.

\bibitem{stern57} S. Sternberg.
\newblock {\em Local Contractions and a Theorem of Poincare. }
\newblock American Journal of Mathematics, Vol. 79, No. 4, pp. 809-824, 1957.

\bibitem{stern58} S. Sternberg.
\newblock {\em On the Structure of Local Homeomorphisms of Euclidean n-Space. }
\newblock American Journal of Mathematics, Vol. 80, No. 3, pp. 623-631, 1958.

% T

\bibitem{takens71} F. Takens. 
\newblock {\em Partially Hyperbolic Fixed Points. }
\newblock Topology Vol.10, 1971. 

% U

\bibitem{ugg13a}  C.~Uggla.
\newblock {\em Spacetime singularities: Recent developments. }
\newblock Int.\ J.\ Mod.\ Phys.\ D\ 22, 2013.

\bibitem{ugg13b}  C.~Uggla.
\newblock {\em Recent developments concerning generic spacelike singularities. }
\newblock Plenary Contribution to ERE2012, http://arxiv.org/abs/1304.6905, 2013.

\bibitem{uggetal03} C.~Uggla,, H.~van Elst, J.~Wainwright and G.F.R.~Ellis.
\newblock {\em The past attractor in inhomogeneous cosmology. }
\newblock Phys.\ Rev.\ D 68, 2003.

% V

\bibitem{vanderbauwhede} Vanderbauwhede, A. 
\newblock {\em Centre Manifolds, Normal Forms and Elementary Bifurcations. }
\newblock Dynamics Reported, 2, 89-169, 1989.

% W

\bibitem{waiell97} J. Wainwright and G.F.R. Ellis.
\newblock {\em Dynamical systems in cosmology. }
\newblock Cambridge University Press, Cambridge, 1997.

\bibitem{waihsu89} J. Wainwright and L. Hsu.
\newblock {\em A dynamical systems approach to Bianchi cosmologies:
orthogonal models of class A. }
\newblock Class.\ Quantum Grav. 6, 1989. 

\bibitem{wald84} R. M. Wald.
\newblock {\em General Relativity. }
\newblock University Of Chicago Press, 1984.

\bibitem{wheeler10} J.A. Wheeler.
\newblock {\em Geons, Black Holes, and Quantum Foam: A Life in Physics. }
\newblock W. W. Norton \& Company, 2010.



%\end{comment}
\end{thebibliography}
\end{document}